%% file: main.tex
\documentclass[11pt,a4paper]{amsart}

\usepackage{geometry}
\usepackage{amsmath}
\usepackage{amsfonts}
\usepackage{amssymb}
\usepackage{amsthm} 
\usepackage{mathtools}

\usepackage{graphicx}
\usepackage{wrapfig}
\usepackage{caption}
\usepackage{subcaption}
\usepackage{tikz-cd} 

\usepackage{booktabs}
\usepackage{hyperref}
\usepackage{enumerate}
\usepackage{xcolor}
\usepackage{url}

\usepackage[foot]{amsaddr}

\graphicspath{{./images/}}

\hypersetup{
	colorlinks=true, 
	linkcolor=black, 
	citecolor=black
}

\DeclareMathOperator{\dist}{dist}

\NewDocumentCommand{\pa}{}{\partial}

\NewDocumentCommand{\Lc}{}{\mathcal{L}}

\NewDocumentCommand{\R}{}{\mathbb{R}}

\NewDocumentCommand{\N}{}{\mathbb{N}}
\NewDocumentCommand{\B}{}{\mathbb{B}}

\NewDocumentCommand{\Sb}{}{\mathbb{S}}

\NewDocumentCommand{\eb}{}{\mathbf{e}}

\theoremstyle{plain}
\newtheorem{thm}{Theorem}[section]
\newtheorem{prop}[thm]{Proposition}
\newtheorem{lem}[thm]{Lemma}
\newtheorem*{lemma*}{Lemma}

\newtheorem{cor}[thm]{Corollary}

\theoremstyle{definition}
\newtheorem{defn}[thm]{Definition}
\newtheorem{rem}[thm]{Remark}
\newtheorem*{rem*}{Remark}
\newtheorem{eg}[thm]{Example}

\numberwithin{equation}{section}

\title{Increase in topological complexity along the mean curvature flow}

\author{Natasha Diederen}
\email{n.diederen24@imperial.ac.uk}

\begin{document}

\begin{abstract}
	We construct an example of a smooth embedding of \(\Sb^{p+q-1}\) in \(\R^{p+q}\) whose evolution under the mean curvature flow forms an isolated singularity at one point, after which it is a smoothly embedded copy of \(\Sb^{p-1}\times \Sb^q\). In doing so, we prove some general results about \(SO(p)\times SO(q)\)-invariant mean curvature flow whose initial data is generated by the rotation of a graphical profile curve, including an analysis of bubblesheet singularities. The most involved part of our analysis is an instant smoothness result, which requires a delicate approximation and pseudolocality. 
\end{abstract}

\maketitle

\section{Introduction}\label{sec:introduction}
\input{0_introduction}
\section{Preliminaries}
\label{sec:preliminaries}
\input{1_preliminaries}
\section{Formation of a bubblesheet singularity}
\label{sec:singularity_formation}
\input{3_singularity_formation}
\section{Isolated pinching at origin}
\label{sec:topological_change_example}
\input{4_isolated_pinching}
\section{Smooth continuation of the flow}
\label{sec:instant_smoothness}
\input{5_instant_smoothness}
\bibliographystyle{alpha}
\bibliography{bibliography}
\end{document}

%% file: 0_introduction.tex
A hypersurface in \(\R^{n+1}\) evolves by mean curvature flow if it moves with velocity equal to its mean curvature vector at each point. While the flow remains smooth, the evolving hypersurfaces are diffeomorphic, and so the topology is preserved by the flow. However, solutions to the mean curvature flow are subject to finite-time singularity formation, and it is at these singular times that the topology may change. A central problem in the study of mean curvature flow lies in understanding the nature of singularities and the ways in which they can alter the topology of the evolving hypersurface. 

In the case of mean curvature flow starting from a smooth, closed, embedded surface in \(\R^3\), it is known that the genus is non-increasing with time \cite{White95}. In other words, the mean curvature flow only serves to simplify the topology of the evolving surface, up to creating connected components. An example of this is Grayson's dumbbell \cite{Grays89}, a mean curvature flow whose initial datum consists of two spheres connected by a thin neck. If the neck is sufficiently long and thin, the surface will pinch off and disconnect into two spheres, each of which then shrinks to a point. In higher dimensions, however, it is possible for the topology of a hypersurface to become more complex after the first singular time. For example, White discusses a \(SO(p)\times SO(q)\)-invariant generalisation of Grayson's dumbbell that starts off diffeomorphic to \(\Sb^{p+q-1}\), but whose \((q-1)\)\textsuperscript{th} homology group becomes non-trivial at some singular time \cite{White95}.  

Although White's example demonstrates that it is possible for an intermediate homology group (i.e. \(H_k\) for \(0<k<n\)) to become non-trivial under the mean curvature flow, the argument is non-constructive, and one can not deduce what the precise topology of the hypersurface is at any time. The expectation is that examples with well controlled topological change exist, but so far none have been rigorously constructed. In this paper, we take White's example as a starting point to construct the simplest possible example of an increase in topology. By this we mean that our singular set consists of one isolated point and the topology changes in a prescribed manner. Doing this requires control over the number, location and nature of singularities, as well as regularity of the flow past singular time. General theory does not at present provide such precise control, so we must perform a detailed analysis using the specific structure of the initial hypersurface. Our main result is as follows.

\begin{thm}\label{thm:main_theorem}
    There exists a smooth, mean convex embedding \(M_0\cong \Sb^{p+q-1}\) in \(\R^{p+q}\) whose evolution under the mean curvature flow has the following properties:
    \begin{enumerate}[(1)]
        \item the flow forms an isolated \(\R^p\times\Sb^{q-1}\) bubble-sheet singularity at the first singular time \(t_0\);
        \item there exists some \(\varepsilon>0\) such that each time slice \(M_{t_0+t}\) is a smooth embedding of \(\Sb^{q}\times \Sb^{p-1}\) for \(0<t < \varepsilon\).
    \end{enumerate}
\end{thm}

Theorem \ref{thm:main_theorem} contributes to the current understanding of the mean curvature flow in the following key ways. Firstly, it is conjectured that singularities of a generic mean curvature flow are isolated in space and time (\cite[Problem 9]{Ilman03}, \cite[Conjecture 7.1]{ColdiP15}), and recent results show that cylindrical singularities in particular are generically isolated \cite{Sun25,Szeke26}. To the best of our knowledge, Theorem \ref{thm:main_theorem} is the first explicit construction of an isolated bubble sheet singularity for non-vanishing flow. Secondly, in the case of mean convex mean curvature flow, it is known that the possible topological changes that occur are restricted by the type of singularity that occurs. In particular if a \(\R^k\times \Sb^{n-k}\) singularity forms and the flow is smooth at some time before and after the singularity, then the inclusion \(\textnormal{int}(M_{t_2})\subset \textnormal{int}(M_{t_1})\) induces an injection from \(H_j(\textnormal{int}(M_{t_2}))\) to \(H_j(\textnormal{int}(M_{t_1}))\) for \(j\geq k\) \cite[Theorem 6.2]{White13}. That is to say, the structure of the tangent flow gives a lower bound above which there can be no increase in \(j\)\textsuperscript{th} Betti number of the interior of the flow. Written in terms of the interiors of the time slices, our example is one where \(\B^{p+q}\) becomes \(\B^{q+1}\times \Sb^{p-1}\) after passing though an \(\R^q\times \Sb^{p-1}\) bubblesheet singularity, and thus demonstrates this bound is sharp whilst undergoing the simplest allowable change in topology\footnote{Theorem 6.2 of \cite{White13} follows from applying Alexander duality to results describing topology change in the complement of mean convex level set flow, White demonstrates that the results used are sharp via an example that vanishes at the first singular time, and therefore does not satisfy the assumptions of Theorem 6.2 itself, although sharpness is to be expected}. 

To keep notation simple, we only prove Theorem \ref{thm:main_theorem} for \(SO(2)\times SO(2)\)-invariant mean curvature flow in \(\R^4\). However, there is nothing conceptually different involved in proving Theorem \ref{thm:main_theorem} in full generality; we make a note of the main modifications required in Remark \ref{rem:higher_dimensions} at the end of this paper. Even viewed purely as an example in \(\R^4\), Theorem \ref{thm:main_theorem} is already of interest since the study of mean curvature flow in \(\R^4\) is substantially more challenging than in \(\R^3\). In particular, it is the first dimension in which bubble-sheet singularities appear, which require much more delicate analysis than neck-pinch or spherical singularities. In terms of the broader picture, Bamler and Lai \cite{Bamle25} recently provided a complete classification of blow-ups at cylindrical singularities in all dimensions, following on from a line of work on the classification of ancient flows in \(\R^4\) \cite{Zhu22,ChoiHaslhHersh23,ChoiHaslhHersh24,Du23,Du24,Choi25,ChoiHaslh24}.

\subsection{Overview}
In this paper we work in the framework of \(SO(p)\times SO(q)\)-invariant mean curvature flow generated by a graphical profile curve. The analysis in this setting shares similarities with the study of \(SO(n)\)-invariant flow of a graphical profile curve carried out by Altschuler, Angenent and Giga \cite{Altsh95}. In their setting, they are able to show that there are finitely many singular times at which the mean curvature flow forms finitely many isolated neck-pinch singularities. In our case, the second direction of rotation introduces additional complexities which are not present in the \(SO(n)\)-case. Consequently, we do not attempt to prove results with the same level of generality, and instead focus on understanding the behaviour of one specific example. We restrict ourselves to mean convex mean curvature flow which makes certain components of our analysis more straightforward, most notably the blow-up analysis near singularities and non-fattening of the level set flow past the first singular time. 

The structure of our paper is as follows. In Section \ref{sec:preliminaries}, we give an account of standard definitions and results we will use throughout the paper, then prove some preliminary results about \(SO(p)\times SO(q)\)-invariant mean curvature flow generated by a graph. In Section \ref{sec:singularity_formation}, we switch to considering mean curvature flow in \(\R^4\) and show that if a mean convex \(SO(2)\times SO(2)\)-invariant mean curvature flow is non-vanishing at the first singular time, then all singularities must lie on one of the planes of rotation and all tangent flows are modelled on a shrinking \(\R^2\times \Sb^1\). An important consequence of this result is that it allows us to say something about the regularity of the flow at the first singular time, in particular that the limiting hypersurface must be generated by the graph of a continuous function that is smooth away from its zero set. 

In Section \ref{sec:topological_change_example}, we demonstrate a general process by which a profile curve can be modified to generate a non-vanishing, mean convex mean curvature flow, so that the results of Section \ref{sec:singularity_formation} can be applied. Moreover, we construct an example that forms an isolated singularity at the first singular time, thus establishing the validity of statement (1) of Theorem \ref{thm:main_theorem}. As mentioned previously, our example shares similarities with the generalised dumbbells described in Example 5 of \cite{White95}. However, in White's construction, a large portion of the initial profile curve is taken to be a flat line, so there is little hope of being able to show that singularities are isolated. In our case, we consider an initial profile curve that has a unique local minimum, and use the Sturmian theorem to show that this structure persists up to the first singular time. Even with this improved structure, a barrier argument is still needed to rule out the possibility that a 1-dimensional portion of the profile curve generates the singular set at the first singular time. 

Finally, in Section \ref{sec:instant_smoothness}, we demonstrate that the example constructed in Section \ref{sec:topological_change_example} becomes smooth immediately after the first singular time. This is the most involved part of the analysis, requiring a careful approximation of the flow at the singular time and different considerations close to and away from the singular point. The proof of Theorem \ref{thm:main_theorem} is then completed.

\subsection{Acknowledgments} 
The author is grateful to Stephen Lynch for suggesting this topic and for many helpful discussions. This work was supported by the London School of Geometry and Number Theory, a joint venture between University College London, Imperial College London and King's College London. 

%% file: 1_preliminaries.tex
We say that a family of smooth hypersurfaces \(\{M_t\}_{t\in I}\) in \(\R^{n+1}\) is a mean curvature flow if
\begin{align}\label{eq:mcf_def}
    \frac{\pa x}{\pa t} = -H(x)\nu(x)
\end{align}
for \(x\in M_t\), \(t\in I\), \(I\subset\R\) an open interval. Here, \(H\) is the scalar mean curvature and \(\nu\) the outward pointing unit normal at a point. The mean curvature flow is a non-linear second order parabolic partial differential equation and, as such, solutions with smooth initial data enjoy properties such as short-time existence and uniqueness \cite{Huisk99a}. Moreover, compact solutions to the mean curvature flow obey an avoidance principle and, by comparison to a shrinking sphere, must form singularities in finite time. A key component of the study of mean curvature flow is the analysis and classification of the types of singularities that can form. A standard tool used to study this is the tangent flow. 

\subsection{Tangent flows}
\label{sec:tangent_flows}
A \emph{tangent flow} \(\tilde{M}_t\) at a point \((x_0,t_0)\) is the limit (as \(s_k\to\infty\)) of a sequence of rescaled flows
\begin{align*}
    \tilde{M}_t^k = s_k\big(M_{s_k^{-2}t+t_0}-x_0\big), 
    \quad t\in[-s_k^2 t_0,\ 0)
\end{align*}
which correspond to translating the point \((x_0,t_0)\) to the origin \((0,0)\) then taking a sequence of parabolic dilations \((x,t)\mapsto (s_kx,s_k^2t)\). For weak notions of mean curvature flow, the above limit is taken in the sense of varifold convergence. However, in special cases, such as when the flow is mean convex, the limit is with respect to local smooth converge. By Huisken's monotonicity formula \cite{Huisk90} and further results due to Ilmanen and White \cite{White94,Ilman95}, it is known that all tangent flows at a singular point are self-similar shrinkers. Tangent flows in the mean convex setting are particularly well understood, as the following result demonstrates. 

\begin{thm}[{\cite{White00,White03,White15}}]\label{thm:white_classification}
    Let \(M_0\subset \R^{n+1}\) an embedded, closed, smooth, mean convex hypersurface, and let \(\{M_t\}\) be the corresponding mean curvature flow. If \(X = (x_0,t_0)\) is a singular point, then each tangent flow at \(X\) is a shrinking sphere or cylinder. In other words, if \(\tilde{M}_t\) is a tangent flow to \(M_t\) at \(X\), then \(\tilde{M}_{-1}\) is (up to rotation) of the form \(\R^k\times \Sb^{n-k}\) for some \(k\in \{0,1,\dots,n-1\}\).
\end{thm}

If the tangent flow at a singular point is a shrinking \(\R\times \Sb^{n-1}\), we call the corresponding singularity a \emph{neck-pinch}, and if the tangent flow is a shrinking \(\R^k\times\Sb^{n-k}\) for \(k\geq 2\), we call the corresponding singularity a \emph{bubble-sheet}. Along with understanding the local behaviour near singularities, we also want to be able to meaningfully extend the flow past a singular time. This requires a weak notion of mean curvature flow. Here, we consider the level set flow. 

\subsection{Level set flow}
\label{sec:level_set_flow}

Introduced independently by Chen-Giga-Goto \cite{Chen91} and Evans-Spruck \cite{Evans91,Evans92a,Evans92b,Evans95}, the level set flow is a viscosity solution to the mean curvature flow. 
For a closed subset \(K\subset \R^{n+1}\), we denote by \(F_t(K)\) the result of letting \(K\) flow for time \(t\) by the level set flow. The level set flow has the following key properties which we make frequent use of (c.f. \cite[Section 2]{White00}):
\begin{enumerate}[(1)]
    \item if \(t\mapsto M_t\) is a family of smooth surfaces moving classically by mean curvature and \(K\) is a set such that \(\pa K = M_0\), then \(\pa F_{t}(K) = M_t\) for all \(t\) before the first singular time;
    \item if \(K\) and \(K'\) are disjoint closed sets with \(K\) compact, then \(F_t(K)\) and \(F_t(K')\) are also disjoint;
    \item there is long time-existence and uniqueness for the flow.
\end{enumerate}

The notion of mean convexity extends well to level set flow. We say that a compact set \(K\subset \R^{n+1}\) is \emph{mean convex} if there exists \(\varepsilon>0\) such that \(F_{h}(K) \subset \textnormal{int}(K)\) for \(0< h < \varepsilon\). If \(K\) is compact and mean convex, then it can be shown that \(F_{t+h}(K) \subset \textnormal{int}(F_t(K))\) for all \(t+h>t\geq 0\). In particular, mean convexity is preserved by the level set flow. If \(\pa K\) is smooth, then this definition of mean convexity is equivalent to saying the mean curvature of the surface is positive everywhere. 

The level set flow starting from the boundary of non-smooth initial initial data may develop a non-empty interior through a process known as \emph{fattening}. However, if the initial data is mean convex, this cannot happen. More formally, if \(K\) is compact and mean convex, then \(F_t(\pa K) = \pa F_t(K)\) \cite[Corollary 3.3]{White00}. In this paper, we will only deal with mean convex level set flow, and will therefore always talk about the level set flow of a hypersurface, rather than its interior. When it is necessary to distinguish between the mean curvature flow and level set flow we will write \(F_t(M_0)\) for the level set flow and \(M_t\) for a smooth mean curvature flow with initial datum \(M_0\). 

The level set flow of a general set \(U\) behaves will with approximations, in the sense that the level set flow of inner and outer approximations\footnote{by an \emph{inner approximation} we mean a collection \(U^j\) such that \(U^j\subseteq U^{j+1}\subseteq U\) and \(U = \bigcup_{j}U^j\), and by an \emph{outer approximation} we mean a collection \(E^j\) such that \(U\subset E^{j+1}\subset E^j\) and \(U = \bigcap_{j}E^j\). We write \(U^j\uparrow U\) in the first case and \(E^j\downarrow U\) in the second} of \(U\) converge to the level set flow of \(U\) \cite[Section 3]{Altsh95}. However, for our purposes, we would like the level set flow of the boundaries \(\pa U^j\) to converge to the level set flow of \(\pa U\) in some sense too. We make this more precise in the following remark, which we rely on for our proof of instant smoothness in Section \ref{sec:instant_smoothness}.

\begin{rem}\label{rem:mean_convex_limit}
    First note that if \(U\) is mean convex then the inner and outer flows coincide and \(F_t(U) = \overline{\bigcup_j F_t(U^j)}\) \cite[Section 3]{Altsh95}. Using in addition that \(F_t(U^j)\uparrow F_t(U)\) as discussed above, a standard topological argument implies that \(\pa F_t(U^j)\) converges to \(\pa F_t(U)\) in the Hausdorff sense. If we additionally take our inner approximation \(U^j\) to be mean convex, then the non-fattening result above tells us that \(F_t(\pa U^j)\to F_t(\pa U)\) in the Hausdorff sense. 
\end{rem}

\subsection{\(SO(p)\times SO(q)\)-invariant mean curvature flow generated by a graph}
\label{sec:so_2_so_2_invariant_flow}
For the remainder of this paper, we restrict our attention to mean curvature flow whose initial data is invariant under the natural action of \(SO(p)\times SO(q)\). We first set up some notation. We denote coordinates in \(\R^{p+q}\) by \((\bar{x},\bar{y})=(x_1,\dots,x_p,y_1,\dots, y_q)\), the canonical basis of \(\R^{p+q}\) by \(\{\eb_1,\eb_2,\dots,\eb_{p+q}\}\), and let
\begin{align*}
    &P_0 = \{(\bar{x},\bar{y})\in \R^{p+q}: x_2 =\dots =x_p = y_2= \dots = y_q = 0\},\\
    &P_1 = \{(\bar{x},\bar{y})\in \R^{p+q}: \bar{y} = 0\}, \quad
    P_2 = \{(\bar{x},\bar{y})\in \R^{p+q}: \bar{x} = 0\}.
\end{align*}
Any hypersurface invariant under the action of \(SO(p)\times SO(q)\) can be generated by the rotation of a profile curve \(\gamma\) lying in \(P_0 \cap \{x_1\geq 0, y_1\geq 0\}\) around the planes \(P_1\) and \(P_2\). Symbolically,
\begin{align*}
    \Sigma_\gamma = \{(\bar{x},\bar{y})\in \R^{p+q}: (|\bar{x}|,\ |\bar{y}|) \in \gamma\}
\end{align*}

We remark that this symmetry is preserved by the level set flow. Indeed, let \(F_t(M_0)\) be a level set flow with \(SO(p)\times SO(q)\)-invariant initial data \(M_0\). For arbitrary \(R\in SO(p)\times SO(q)\) the rotation \(\{R(M_t)\}_{t<t_0}\) is still a level set flow. On the other hand, \(M_0 = R(M_0)\) by assumption, so it follows from the long-time existence and uniqueness of the level set flow that the two flows coincide. In particular, the preservation of symmetry means that up to the first singular time, where the level set flow and smooth mean curvature flow coincide, the mean curvature flow of \(M_t\) can be entirely determined by the evolution of a 1-parameter family of profile curves \(\gamma_t\).

From now on we make the additional assumption that the profile curve \(\gamma\) is a graph \(\gamma_u \coloneqq \{(x,0,u(x),0):x\in[\xi_1,\xi_2]\}\), where \(u\in C^0([\xi_1,\xi_2], \R_{\geq 0})\) and \(0\leq \xi_1 < \xi_2 < \infty\). We denote the corresponding hypersurface by \(\Sigma_u = \Sigma_{\gamma_u}\). We say that \(\gamma_u\) \emph{crosses \(P_2\)} if \(\xi_1 = 0\) and \emph{avoids \(P_2\)} if \(\xi_1>0\). Similarly \(\Sigma_u\) crosses (or avoids) \(P_2\) if it is generated by a profile curve \(\gamma_u\) that does so. By imposing further conditions on \(u\), we can ensure that \(\Sigma_u\) is smooth and closed. In particular, the graph of \(u\) needs to meet the \(x_1\)- and \(y_1\)-axes orthogonally. This corresponds to the condition that
\begin{align}\label{eq:avoiding_P2_smoothness}
    u(x) > 0 \textnormal{ for } x\in (\xi_1,\xi_2), \ u(\xi_1) = u(\xi_2) = 0, \ u_x(\xi_1) = \infty \textnormal{ and } u_x(\xi_2) = -\infty
\end{align}
in the case where \(\Sigma_u\) avoids \(P_2\), and
\begin{align}\label{eq:crossing_P2_smoothness}
    u(x) > 0 \textnormal{ for } x\in [0,\xi_2), \ u(\xi_2) = 0, \ u_x(0) = 0 \textnormal{ and } u_x(\xi_2) = -\infty
\end{align}
in the case where \(\Sigma_u\) crosses \(P_2\). Note that if \(\Sigma_u\) is smooth, closed and crosses \(P_2\), then \(\Sigma_u\cong \Sb^{p+q-1}\), whereas if \(\Sigma_u\) is smooth, closed and avoids \(P_2\), then \(\Sigma_u\cong \Sb^{p-1}\times \Sb^q\). This fully classifies the possible topological structures of \(\Sigma_u\). 

Whenever the expressions are well-defined, the principal curvatures of \(\Sigma_u\) are
\begin{align}\label{eq:principal_curvatures}
    \kappa_1 = \frac{-u_{xx}}{(1+u_x^2)^\frac{3}{2}}, 
    \quad 
    \kappa_{2} = \dots = \kappa_{q} = \frac{1}{(1+u_x^2)^\frac{1}{2}u},
    \quad
    \kappa_{1+q} = \dots = \kappa_{n} = \frac{-u_x}{(1 + u_x^2)^\frac{1}{2}x},
\end{align}
where our sign convention is such that the principal curvatures of the sphere are positive. As a slight abuse of notation these are not in order of size, as is standard. Instead, \(\kappa_1\) is the curvature of the curve, \(\kappa_2,\dots,\kappa_{q}\) are obtained by rotation around \(P_1\) and \(\kappa_{1+q},\dots,\kappa_n\) are obtained by rotation around \(P_2\). Note that in order for the principal curvatures \eqref{eq:principal_curvatures} to be defined everywhere along the profile curve (except possibly at its end-points) we require that \(u\in C^0([0,\xi_2])\cap C^\infty([0,\xi_2))\) in the case where \(\Sigma_u\) crosses \(P_2\) and \(u\in C^0([\xi_1,\xi_2])\cap C^\infty((\xi_1, \xi_2))\) in the case where \(\Sigma_u\) avoids \(P_2\). For the sake of concision, for the remainder of this paper we say that \(M_0\) satisfies hypothesis:
\begin{enumerate}[(H1)]
    \item if \(M_0 = \Sigma_u\) is a smooth, closed \(SO(p)\times SO(q)\)-invariant hypersurface avoiding \(P_2\) generated by the graph of some \(u\in C^0([\xi_1,\xi_2])\cap C^\infty((\xi_1,\xi_2))\) satisfying conditions \eqref{eq:avoiding_P2_smoothness};
    \item if \(M_0 = \Sigma_u\) is a smooth, closed \(SO(p)\times SO(q)\)-invariant hypersurface crossing \(P_2\) generated by the graph of some \(u\in C^0([\xi_1,\xi_2])\cap C^\infty((\xi_1,\xi_2))\) satisfying conditions \eqref{eq:crossing_P2_smoothness}.
\end{enumerate}

We now want to understand the evolution of \(M_0 = \Sigma_u\) under the mean curvature flow. It can be shown through a direct computation that if the evolution of \(M_t\) remains represented by a profile curve that is the graph of some function \(u\), then \(u\) must evolve according to 
\begin{align}\label{eq:evolution_equation}
    \frac{\pa u}{\pa t} = \frac{u_{xx}}{1+u_x^2} - \frac{q-1}{u} + \frac{(p-1)u_x}{x},
\end{align}
where \(u_t = u(\cdot,t):[\xi_1(t),\xi_2(t)]\to \R_{\geq 0}\) is a 1-parameter family of functions such that \(M_t = \Sigma_{u_t}\).  The following result tells us that \(u\) is a smooth function away from its end-points up to the first singular time, and so the evolution equation \eqref{eq:evolution_equation} is well-defined up to this time. 

\begin{prop}\label{prop:graphicality}
    Suppose that \(M_0\) satisfies hypothesis (H1) or (H2), and let \(\{M_t\}_{t< t_0}\) be the corresponding smooth mean curvature flow. Up to the first singular time \(t_0\), \(M_t\) remains a smooth \(SO(p)\times SO(q)\)-invariant hypersurface generated by some \(u_t\in C^0([\xi_1(t),\xi_2(t)])\cap C^\infty((\xi_1(t),x_2(t)))\) or \(u_t\in C^0([0,\xi_2(t)])\cap C^\infty((0,x_2(t)))\).
\end{prop}
\begin{proof}
    Due to the preservation of \(SO(p)\times SO(q)\) symmetry, the mean curvature flow \(\{M_t\}\) can be described by the evolution of a 1-parameter family of profile curves \(\gamma(\cdot,t)\). If we can rule out vertical tangents (i.e. tangents parallel to \(\eb_{p+1}\)) from forming at any time up to the singular time and away from \(P_1\), then the desired result follows from the implicit function theorem. Let \(\nu\) be the outward normal to \(M_t\). Its restriction to \(P_0\) induces a normal on \(\gamma\) which we denote by \(\tilde{\nu}\). Vertical tangents occur precisely when \(\tilde{\nu}\cdot \eb_{p+1} = 0\). To rule this out, it suffices to show that \(\nu \cdot\eb_{p+1}>0\) on \(\hat{M}_t \backslash \pa \hat{M}_t\), where \(\hat{M}_t \coloneqq M_t\cap \{y_1>0\}\), as this then implies that \(\tilde{\nu}\cdot \eb_{p+1}>0\) as well. Our assumptions on \(M_0\) give that \(\nu\cdot \eb_{p+1}\geq 0\) on \(\hat{M}_0\), with equality holding on \(\pa \hat{M}_0\) only. The smoothness of the flow up to the first singular time implies that \(\nu\cdot \eb_{p+1} \geq 0\) for \(0\leq t< t_0\). Now, using that \(\nu\cdot \eb_{p+1}\) satisfies \((\pa_t - \Delta)(\nu\cdot \eb_{p+1}) = |A|^2(\nu\cdot \eb_{p+1})\), it follows from the weak maximum principle that \(\nu\cdot \eb_{p+1} \geq 0\) on \(\hat{M}_t\), and by the strong maximum principle that \(\nu\cdot \eb_{p+1} > 0\) on \(\hat{M}_t\backslash \pa \hat{M}_t\).
\end{proof}

As a result of Proposition \ref{prop:graphicality}, we are able to reduce studying the mean curvature flow of a hypersurface to studying the solution \(u(\cdot,t)\) to a one-dimensional parabolic partial differential equation \eqref{eq:evolution_equation}. This enables us to make use of the following Sturmian theorem, which will later allow us to control the number of minima and maxima of our evolving profile curve. 

\begin{thm}[{\cite[Section 2]{Angen91}}]\label{thm:sturmian_theorem}
    Let \(f = f(x,t,u,p,q)\) be a smooth function of its five variables, and let \(u\), \(v\) be two smooth solutions of the non-linear equation
    \begin{align}\label{eq:sturmian}
        \frac{\pa u}{\pa t} = f(x,t,u,u_x,u_{xx}).
    \end{align}
    on \([x_1,x_2]\times [0,t_0]\). Assuming that \eqref{eq:sturmian} is parabolic, if \(u(x_j,t)\neq v(x_j,t)\) for \(0\leq t \leq t_0\), and \(j=1,2\), the number of zeros of \(x\mapsto v(x,t)-u(x,t)\), when counted with multiplicity\footnote{the number of zeros of \(f\) in \((a,b)\) counted with multiplicity is the sum of the orders of all zeros of \(f\) in \((a,b)\), where the order of a zero at \(x_0\) is the smallest integer \(k\) for which \(f^{(k)}(x_0)\neq 0\)}, is finite for any \(t\in (0,t_0]\), and does not increase with time. Moreover, at any time \(t\) where the zero is of order \(k>1\), the number of zeros at time \(t\) drops by at least \(k-1\). 
\end{thm}
\begin{rem}\label{rem:curvilinear}
    It is possible to apply Theorem \ref{thm:sturmian_theorem} on a non-rectangular domain, as is done in Section 2.1 of \cite{Coicu24}. Specifically, it is possible to replace \(\{x_1,x_2\}\times(0,t_0]\) with \(\{x_1(t),x_2(t)\}\times(0,t_0]\), where \(x_1\) and \(x_2\) are analytic in \(t\), and \(u(x_j(t),t)\neq v(x_j(t),t)\) for \(0\leq t<t_0\).
\end{rem}

From now on we restrict our attention to mean curvature flow in \(\R^4\) and fix \(p=q=2\).

%% file: 3_singularity_formation.tex
In this section, we consider mean curvature flow with initial data \(M_0=\Sigma_u\) that is mean convex and satisfies hypothesis (H1) or (H2). Assuming that the mean curvature flow \(\{M_t\}_{t\leq t_0}\) does not vanish at the first singular time \(t_0\), we analyse the singularities that form. As a consequence of this analysis, we are able to conclude that the limiting hypersurface at the first singular time is is generated by a continuous graphical profile curve that is smooth away from its zero set.

\subsection{Blow-up analysis}

We start by analysing the possible tangent flows at singularities. Since we do not yet know anything about the regularity of the hypersurface at the singular time, we must first set up some terminology to describe the location of points on the hypersurface without referring to a profile curve. We say that a mean curvature flow \(\{M_t\}_{t<t_0}\) reaches \(x_0\) at time \(t_0\) if there exists a sequence \((x_j,t_j)\) with \(t_j \to t_0\) such that \(x_j\in M_{t_j}\) and \(x_j\to x_0\). Since the level set flow preserves \(SO(p)\times SO(q)\)-symmetry, we can assume without loss of generality that a point \(x_0\) reached by the flow at the first singular time \(t_0\) lies in \(P_0\). We write \(M_{t_0}\) for the limiting hypersurface. 

By mean convexity of the flow, the points of the profile curve that lie on \(P_1\) and \(P_2\) move monotonically with time, so the monotonic convergence theorem can be applied to deduce a limit exists at time \(t_0\). If \(M_t\) avoids \(P_2\) for all time, we say that \(x_0 \in M_{t_0}\) is a \emph{left} or \emph{right \(P_1\)-crossing point} if it is the limit as \(t\to t_0\) of the end-points \(\xi_1(t)\) or \(\xi_2(t)\) respectively. If \(M_t\) crosses \(P_2\), then we will no longer have a left \(P_1\)-crossing point, but will instead have a \emph{\(P_2\)-crossing point}, which we define as the limit as \(t\to t_0\) of \(u(0,t)\). The definition of the right \(P_1\)-crossing point remains the same in this case. The main result of this section is as follows. 

\begin{prop}\label{prop:bubblesheet}
    Suppose that \(M_0 = \Sigma_u\) is mean convex, satisfies hypothesis (H1) or (H2) and the corresponding level set flow \(\{M_t\}_{t>0}\) does not vanish at the first singular time \(t_0\). Suppose further that \(X_0 = (x_0,t_0)\) is a singular point with \(x_0\in P_0\). Then \(x_0\) must lie on \(P_1\) but not at a \(P_1\)-crossing point of \(M_{t_0}\). Moreover, the tangent flow at \(X_0\) is a shrinking \( \R^2\times \Sb^1\), with the \(\Sb^1\) factor parallel to \(P_2\).
\end{prop}

We start by proving the following intermediate result.

\begin{lem}\label{lem:disc_containment}
    Suppose \(M_0 = \Sigma_u\) satisfies the assumptions of Proposition \ref{prop:bubblesheet}, and let \(x_0 = (x^*,0,y^*,0)\in P_0\) be a point reached by the level set flow at the first singular time \(t_0\). Define
    \begin{align*}
        \ell &= \{(\bar{x},\bar{y})\in P_1:x_1 = 0\}, \
        &H_2^- &= \{(\bar{x},\bar{y})\in P_2: y_1 \leq 0\}, \\
        H_1^- &= \{(\bar{x},\bar{y})\in P_1:x_1\leq 0\}, \
        &H_1^+ &= \{(\bar{x},\bar{y})\in P_1:x_1\geq 0\}.
    \end{align*}
    The following statements hold:
    \begin{enumerate}[(a)]
        \item  if \(x^* > 0\) and \(y^*\geq 0\), then the tangent flow at \(X_0\) contains \(\ell\) in its interior;
        \item if \(y^*> 0\) and \(x^*\geq 0\), then the tangent flow at \(X_0\) contains \(H_2^-\) in its interior;
        \item if \(x^* >  0\) and \(y^* = 0\), then
        \begin{enumerate}[(i)]
            \item if \(M_0\) crosses \(P_2\), or if \(M_0\) avoids \(P_2\) and \(x_0\) is not the left \(P_1\)-crossing point of \(M_{t_0}\), then the tangent flow at \(X_0\) contains \(H_1^-\) in its interior;
            \item if \(M_0\) avoids \(P_2\) and \(x_0\) is the left \(P_1\)-crossing point of \(M_{t_0}\), then the tangent flow at \(X_0\) contains \(H_1^+\) in its interior.
        \end{enumerate}
    \end{enumerate}
\end{lem}

\begin{figure}[h]
    \centering
    \includegraphics[width = 0.85\textwidth]{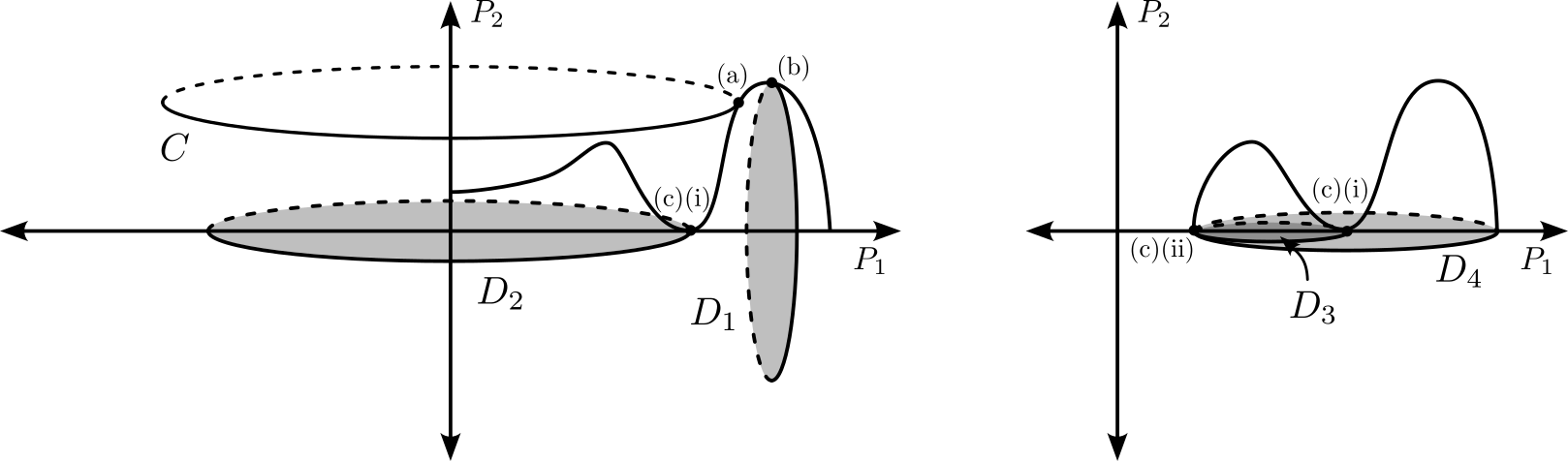}
    \caption{A schematic drawing of the discs used in the proof of Lemma \ref{lem:disc_containment}, horizontal discs should be seen as lying in a plane parallel to \(P_1\) and vertical discs as lying in a plane parallel to \(P_2\)}
    \label{fig:disc_containment}
\end{figure}

\begin{proof}
    The argument in each case follows the same approach. Therefore, we prove (a) in full detail, then point out relevant modifications for the remaining parts.
    \begin{enumerate}[(a)]
        \item By its rotational symmetry around \(P_2\), \(M_{t_0}\) contains a circle \(C\) of radius \(x^*>0\)  centred at \((0,0,y^*,0)\) and parallel to \(P_1\). Let \(\tilde{C}^k = \{s_k(x-x_0):x\in C\}\), then by mean convexity of the flow, it follows that \(\tilde{C}^k\subset \tilde{M}_t^k\) for all \(t< t_0\). Note that \(\tilde{C}^k\) is a circle of radius \(s_kx^*\) with one point passing through the origin and all other points having negative \(x_1\)-coordinate. It follows that \(\tilde{C}^k\to \ell\) in a local smooth sense. As containment is preserved by local smooth convergence, it follows that \(\ell\subset \tilde{M}_t\). 
        \item Since \(M_t\) is generated by a graphical profile curve up to the singular time, it follows that the interior of \(M_{t_0}\) contains the disc \(D_1\) of radius \(y^*>0\) centred at \((x^*,0,0,0)\) and parallel to \(P_2\). The argument follows as before.
        \item 
        \begin{enumerate}[(i)]
            \item If the hypersurface crosses \(P_2\), we apply the same argument as before with the disc \(D_2\) of radius \(x^*>0\) centred at 0 and lying in \(P_1\). If the hypersurface avoids \(P_2\), we apply this argument with the disc \(D_3\) lying in \(P_1\) and touching \(M_{t_0}\) only at the left \(P_1\)-crossing point and at \(x_0\).  
            \item  We apply the same argument as before to the disc \(D_4\) which touches \(M_{t_0}\) at the left and right \(P_1\)-crossing points. \qedhere
        \end{enumerate}
    \end{enumerate}
\end{proof}

Since \(M_0\) is  a connected hypersurface in \(\R^4\), it follows from Theorem \ref{thm:white_classification} that the possible tangent flows at a singular point are a shrinking \(\Sb^3\), \(\R\times \Sb^2\) or \(\R^2\times \Sb^1\). Assuming that the mean curvature flow does not vanish at the first singular time, we can rule out the shrinking \(\Sb^3\) as a possible tangent flow, leaving only the last two options. This classification combined with Lemma \ref{lem:disc_containment} and the \(SO(2)\times SO(2)\) symmetry allows us to rule out the possibility of singularities forming away from \(P_1\) and \(P_2\), or at crossing points (except the origin).

\begin{lem}\label{lem:away_from_axes}
    Suppose \(M_0\) satisfies the assumptions of Proposition \ref{prop:bubblesheet} and \(X_0 = (x_0,t_0)\) is a singular point with \(x_0\in P_0\). Then \(x_0\) must lie on \(P_1\) or \(P_2\).
\end{lem}
\begin{proof}
    Suppose for a contradiction that \(x_0\) is of the form \((x^*,0,y^*,0)\) with \(x^*\neq 0\) and \(y^*\neq 0\). By statements (a) and (b) of Lemma \ref{lem:disc_containment}, the tangent flow \(\tilde{M}_t\) at \(X_0\) must contain \(\ell \subset P_1\) and \(H_2^-\subset P_2\). Since \(P_1\) and \(P_2\) are orthogonal, it follows that \(\tilde{M}_t\) contains \(\ell\times H_2^- = \{(\bar{x},\bar{y})\in \R^4:x_1 = 0,y_1\leq 0\}\) in its interior. However, we also know that \(\tilde{M}_t\) must be a shrinking \(\R\times \Sb^2\) or \(\R^2 \times \Sb^1\), neither of which contain \(\ell\times H_2^-\). This contradicts that \(X_0\) is a singular point. 
\end{proof}

\begin{lem}\label{lem:away_from_endpoints}
    Suppose \(M_0\) satisfies the conditions of Proposition \ref{prop:bubblesheet} and \(X_0 = (x_0,t_0)\) is a singular point with \(x_0\in P_0\). Then \(x_0\) cannot lie at any crossing point of \(M_{t_0}\), except the origin.
\end{lem}
\begin{proof}
    We show this only in the case of a right \(P_1\)-crossing point of \(M_{t_0}\). The proofs in the case of a \(P_2\)-crossing point and a left \(P_1\)-crossing point are analogous (using (b) and (c)(ii) of Lemma \ref{lem:disc_containment} respectively instead of (c)(i)).

    Suppose \(x_0 = (x^*,0,0,0)\) lies at a right \(P_1\)-crossing point (so \(x^*\neq 0\)). Up to the first singular time, the outward normal vector to the surface at \(x_0\) points in the \(\eb_1\) direction, so the tangent vector to the profile curve points in the \(\eb_3\) direction, and this property is preserved by the local smooth convergence when we extract a tangent flow at \(X_0\). Therefore, parabolic rescaling of the profile curve gives a copy of \(\{(\bar{x},\bar{y})\in \R^4: x_1=0\}\) in \(P_0\). By the rotational symmetry around \(P_1\) (which is preserved by the rescaling), rotation of this line gives all of \(P_2\), so if \(X_0\) is singular, then its tangent flow must be a shrinking \(\R^2\times \Sb^1\), with the \(\Sb^1\) factor orthogonal to \(P_2\), i.e. in \(P_1\). However, statement (c)(i) of Lemma \ref{lem:disc_containment} implies that the tangent flow at \(X_0\) contains \(H_1^-\), which cannot be contained in \(\Sb^1\). This is a contradiction.
\end{proof}

We now prove Proposition \ref{prop:bubblesheet}.

\begin{proof}[Proof of Proposition \ref{prop:bubblesheet}]
    Lemmas \ref{lem:away_from_axes} and \ref{lem:away_from_endpoints} immediately imply \(x_0\) lies on \(P_1\), but cannot be a \(P_1\)-crossing point. This proves the first claim. Let us now turn to the second claim, that the tangent flow at \(X_0\) is \(\R^2\times \Sb^1\) with \(\Sb^1\subset P_2\). We split our proof of this claim into the cases \(x_0\neq 0\) and \(x_0 = 0\).

    If \(x_0\neq 0\), then in the plane \(P_2\), the translated flow \(M_t-x_0\) consists of concentric circles centred at 0 of radius \(r(t)\), where \(r(t)\to 0\) as \(t\to t_0\). Consequently, the \(\tilde{M}_t^k\cap P_2\) consists of concentric circles of radius \(s_kr(s_k^{-2}t+t_0)\), still centred at \(0\). By the avoidance principle it follows that
    \begin{align*}
        \dist\left(M_{s_k^{-2}t+t_0},x_0\right) \leq \sqrt{6s_k^{-2}t}.
    \end{align*}
    In particular \(s_kr(s_k^{-2}t+t_0)\leq \sqrt{6t}\), so the rescaled radii are uniformly bounded for each \(t\). It follows that there exits \(r^*(t)\) for which \(s_kr(s_k^{-2}t+t_0)\to r^*(t)\). Therefore, the intersection of the tangent flow with \(P_2\) contains a copy of \(\Sb^1\). By our previous discussion, the tangent flow at \(X_0\) must be a shrinking \(\R\times\Sb^2\) or \(\R^2\times\Sb^1\). Since we are working in the case where \(x^*>0\), and we ruled out \(x_0\) being a left \(P_1\)-crossing point in the first half of the proposition, we can apply statement (c)(i) of Lemma \ref{lem:disc_containment} to conclude that the tangent flow at \(x_0\) contains \(H_1^-\). This cannot be contained in the interior of \(\R\times\Sb^2\), so our tangent flow must be a shrinking \(\R^2\times\Sb^1\).

    If \(x_0=0\), then by the same proof as above, the intersection of our tangent flow with \(P_2\) is a copy of \(\Sb^1\), so the tangent flow at \(X_0\) must either be a shrinking \(\R^2\times\Sb^1\) or \(\R\times\Sb^2\). Assume for a contradiction that it is a shrinking \(\R\times \Sb^2\). Since the intersection of the tangent flow with \(P_2\) contains \(\Sb^1\), the Euclidean factor must be perpendicular to \(P_2\). But then, by the rotational symmetry around \(P_2\) (which is preserved since there is no translation in this case), it follows that \(\R^2\) must be contained in the interior of \(\R\times\Sb^2\), which is a contradiction.
\end{proof}

\subsection{Structure of the limiting hypersurface}

An important consequence of Proposition \ref{prop:bubblesheet} is that it can be used to deduce some information about the structure of the limiting hypersurface at the first singular time \(t_0\). We first note for any \(x\in (\xi_1(t_0),\xi_2(t_0))\), the limit \(u(x,t_0) \coloneqq \lim_{t\to t_0}u(x,t)\) exists due to the mean convexity of the flow. This follows directly from the monotone convergence theorem, since \(u(x,t)\) is monotone decreasing in \(t\) for fixed \(x\). Since the \(SO(2)\times SO(2)\)-invariance is preserved at the singular time, it follows that \(M_{t_0}\) is generated by \(u(\cdot,t_0)\). We now show that the function \(u(\cdot,t_0)\) is `flat' near singularities. An analogous result was established for \(SO(n)\)-invariant flows in \(\R^{n+1}\) \cite[Corollary 3.5]{Altsh95}, and our proof proceeds similarly.

\begin{lem}\label{lem:singularities_are_flat}
    Suppose \(M_0 = \Sigma_u\) satisfies the assumptions of Proposition \ref{prop:bubblesheet}, and let \(\xi\) be such that \(u(\xi,t_0) = 0\), excluding the endpoints \(\xi_1(t_0)\) and \(\xi_2(t_0)\). For every \(\varepsilon>0\) there exists \(\mu_\varepsilon\) such that \(u(x,t_0)\leq \varepsilon |x-\xi|\) for \(|x-\xi|\leq \mu_\varepsilon\). 
\end{lem}
\begin{proof}
    Recall that at a singular point \(X_0=(x_0,t_0)\) of mean convex mean curvature flow, any sequence of rescalings \(\tilde{M}_t^{k} = s_k(M_{s_k^{-2}t + t_0}-x_0)\) converges, up to a subsequence, smoothly on any compact sub-interval of \((-\infty,0)\) to a shrinking sphere or cylinder (c.f. Section \ref{sec:tangent_flows}). In Proposition \ref{prop:bubblesheet}, we established that for \(M_0\) satisfying the conditions of Proposition \ref{prop:bubblesheet}, the tangent flow at a singularity can only be a shrinking \(\R^2\times\Sb^1\), with the \(\Sb^1\) factor parallel to \(P_2\). Since this is true for any sequence of rescalings, we obtain that for every fixed \(t'<0\), 
    \begin{align*}
        \lambda (M_{t_0 + \lambda^{-2}t'} - x_0) \to \sqrt{-t'} \left(\R^2\times \Sb^1_{\sqrt{2}}\right),
    \end{align*}
    smoothly as \(\lambda \to \infty\). Choosing \(\lambda = \left(\frac{-t'}{t_0-t}\right)^{1/2}\) (so \(\lambda\to \infty\) as \(t\to t_0\)) it follows that
    \begin{align*}
        \frac{1}{\sqrt{t_0-t}}(M_t - x_0)\to \R^2 \times \Sb^1_{\sqrt{2}}
    \end{align*}
    smoothly as \(t\to t_0\). Moreover, since \(P_0\) intersects \(\R^2 \times \Sb^1_{\sqrt{2}}\) transversely, the local smooth convergence still holds when both sides are intersected with \(P_0^+\coloneqq P_0 \cap \{(\bar{x},\bar{y})\in \R^4 :x_1\geq 0, \ x_3\geq 0\}\). We can write the convergence in terms of \(u\) as
    \begin{align*}
        \left\{\frac{1}{\sqrt{t_0-t}}((x,0,u(x,t),0) - (\xi,0,0,0)):x\in (\xi_1(t_0),\xi_2(t_0))\right\} 
        \to \left\{(\bar{x},\bar{y})\in P_0^+: \ y_1 \equiv \sqrt{2}\right\}.
    \end{align*}
    Performing the change of variables \(z = (x-\xi)/\sqrt{t_0-t}\), it follows that 
    \begin{align}\label{eq:blowup_limit}
        \lim_{t\to t_0} \frac{u(\xi + z\sqrt{t_0-t},t)}{\sqrt{t_0-t}} = \sqrt{2}
    \end{align}
    uniformly in \(z\) on any bounded interval \(|z|\leq M\).
    
    Now fix \(\varepsilon>0\) and let \(M = \frac{2\sqrt{2}}{\varepsilon}\). There exists \(t_{\varepsilon}<t_0\) such that 
    \begin{align*}
        u(\xi + z\sqrt{t_0-t},t) \leq 2\sqrt{2(t_0-t)}
    \end{align*}
    for \(|z|\leq M\), and \(t_\varepsilon < t < t_0\). Let \(\mu_\varepsilon = M \sqrt{t_0-t}\), then any \(x\neq \xi\) with \(|x - \xi|\leq \mu_\varepsilon\) can be written as \(x = \xi\pm M\sqrt{t_0-t}\) for \(t_\varepsilon < t< t_0\). The mean convexity of the flow implies that \(u(\cdot,t)\) moves strictly downwards with time, and so we find that for such \(x\),
    \begin{equation*}
        u(x,t_0) 
        < u(x,t) 
        \leq 2\sqrt{2(t_0-t)} 
        = 2\sqrt{2}|x-\xi|M^{-1}
        = \varepsilon |x-\xi|. 
    \end{equation*}
    If, on the other hand, \(x = \xi\), then the above inequality is trivially true as \(u(\xi,t_0)=0\). 
\end{proof}

Finally, we show that at the singular time our profile curve is the graph of a continuous function that is smooth away from its zero set. 

\begin{lem}\label{lem:continuity}
    Suppose \(M_0= \Sigma_u\) satisfies the assumptions of Proposition \ref{prop:bubblesheet}, then \(u(\cdot,t_0)\) is a continuous function that is smooth away from its zero set \(Z_{t_0} = \{\xi:u(\xi,t_0) = 0 \}\).
\end{lem}

\begin{proof}
    We first show smoothness. We know from Proposition \ref{prop:bubblesheet} that \(M_{t_0}\) is smooth away from \(P_1\), so we can apply the strong maximum principle to \(\nu\cdot \eb_3\) to conclude that the corresponding profile curve has no vertical tangents, as in the proof of Proposition \ref{prop:graphicality}. Away from \(P_1\), this profile curve can be represented as a smooth graph over the \(x_1\)-axis by the implicit function theorem, which must correspond to the graph of \(u(\cdot, t_0)\) by uniqueness. 

    Continuity at the end-points follows from the fact that, in a neighbourhood of each end-point, the profile curve admits a smooth graphical representation over the \(y_1\)-axis. The image of this graph must agree with the representation of the curve as a graph over the \(x_1\)-axis in some overlapping region, and therefore determines the same limit at the end-point, ensuring continuity there. Finally, for any point \(\xi\in Z_{t_0}\) that is not an end-point, it follows from Lemma \ref{lem:singularities_are_flat} that \(\lim_{x\to\xi} u(x,t_0) = 0\). Since \(u(\xi,t_0) = 0\), continuity follows.
\end{proof}

%% file: 4_isolated_pinching.tex
In this section we prove statement (1) of Theorem \ref{thm:main_theorem}. That is, we show that there exists a smooth, mean convex embedding \(M_0\cong\Sb^3\) in \(\R^4\) whose evolution under the mean curvature flow forms an isolated bubblesheet singularity at the first singular time. We do so by first finding conditions on the initial profile curve under which the flow is mean convex and non-vanishing, which will allow us to apply the results from Section \ref{sec:singularity_formation}. We then use this result coupled with an application of the Sturmian theorem to show that under further conditions on the initial profile curve, an isolated singularity must form at the origin. We end with an explicit construction of a suitable profile curve. 

We begin by showing that given a sufficiently regular function \(u\), one can easily modify the graph of \(u\) so that \(\Sigma_u\) is mean convex. This greatly simplifies the construction of our example as it removes the need to specify an explicit profile curve and calculate its principal curvatures. Recall from \eqref{eq:principal_curvatures} that the principal curvatures of \(\Sigma_u\) are entirely determined by \(u\) and its derivatives, and, moreover \(\kappa_2\) is always positive. The idea of the following result is that if \(\Sigma_u\) is not already mean convex, then stretching the profile curve along the \(x_1\)-axis with large enough scale factor forces \(\kappa_2\) to dominate any negative curvatures. 

\begin{lem}\label{lem:mean_convexity}
    Suppose that \(M_0 = \Sigma_u\) satisfies hypothesis (H2) and assume that there exists a point \(0<p<\xi_2\) such that \(u_x, \ u_{xx} < 0\) for \(x\geq p\). Then there exists \(\lambda\geq 1\) such that \(u_\lambda: y\mapsto u(y/\lambda)\) generates a mean convex hypersurface of rotation \(\Sigma_{u_\lambda}\).
\end{lem}

\begin{proof}
    For \(\lambda \geq 1\), consider the stretched hypersurface of revolution \(\Sigma_{u_\lambda}\), whose profile curve is the graph of \(u_{\lambda}\), for \(y\in [0,\lambda p]\). Using \eqref{eq:principal_curvatures}, the principal curvatures of \(\Sigma_{u_\lambda}\) scale as 
    \begin{align*}
        \kappa_1 = - \frac{u_{xx}(\frac{y}{\lambda})/\lambda^2}{\left(1 + |u_x(\frac{y}{\lambda})|^2/\lambda^2\right)^\frac{3}{2}}, &&
        \kappa_2 = \frac{1}{\left(1+|u_x(\frac{y}{\lambda})|^2/\lambda^2\right)^\frac{1}{2}u(\frac{y}{\lambda})}, &&
        \kappa_3 = -\frac{u_{x}}{\left(1 + |u_x(\frac{y}{\lambda})|^2/\lambda^2\right)^\frac{1}{2}y}.
    \end{align*}
    If \(y\in [\lambda p, \ \lambda \xi_2]\), the principal curvatures at points in the corresponding portion of \(\Sigma_{u_\lambda}\) are all strictly positive by assumption, so we have convexity in this region for any \(\lambda\). 
    
    If \(y\in[0,\lambda p]\), then the smoothness of \(\Sigma_u\) (in particular conditions \eqref{eq:crossing_P2_smoothness}) imply that there exists \(C\) such that 
    \begin{align}\label{eq:derivative_assumptions}
        \sup_{x\in [0, \ p]} \left(|u| + |u_x| + \frac{|u_x|}{|x|}+ |u_{xx}|\right) \leq C,
    \end{align}
    which yields the lower bounds
    \begin{align*}
        \kappa_1 \geq -\frac{C}{\lambda^2}, &&
        \kappa_2 \geq \frac{1}{(1+C^2)^\frac{1}{2}C}, &&
        \kappa_3 \geq -\frac{C}{\lambda^2}.
    \end{align*}
    For any \(C\), we can find \(\lambda\geq 1\) large enough that \(\kappa_1+\kappa_2+\kappa_3 > 0\), and so \(\Sigma_{u_\lambda}\) is mean convex.
\end{proof}

Next, we find inner and outer barriers that ensure that that the mean curvature flow does not vanish at the first singular time, provided the initial data \(M_0\) lies between these barriers. For the outer barrier, we make use of the following hyperboloid comparison principle, which is similar to Proposition 3.7 in \cite{Ecker04}. The key difference being that Ecker seeks to compare the mean curvature flow with an \(SO(n)\)-invariant hyperboloid, whereas we do so with an \(SO(2)\times SO(2)\)-invariant hyperboloid. The proofs are almost identical, so we offer a brief sketch here, and note that more detailed calculations can be found in Chapter 3 of \cite{Ecker04}. 

\begin{lem}\label{lem:hyperboloid_comparison}
    Let \(\{M_t\}_{t>0}\) be a solution of mean curvature flow in \(\R^4\). If for some \(2/3\leq \beta \leq 1\) and some \(\varepsilon \geq 0\) the initial hypersurface \(M_0\) satisfies
    \begin{align*}
        M_0 \subset \{(\bar{x},\bar{y})\in \R^4: (1-\beta)(x_1^2 + x_2^2) \geq y_1^2 + y_2^2 - \varepsilon^2\}
    \end{align*}
    then
    \begin{align*}
        M_t \subset \{(\bar{x},\bar{y})\in \R^4: (1-\beta)(x_1^2 + x_2^2) \geq y_1^2 + y_2^2 + \beta t - \varepsilon^2\}
    \end{align*}
    as long as the solution stays smooth. If in addition \(\varepsilon > 0\) and \(M_0 = \Sigma_u\) satisfies hypothesis (H2), then a singularity has to form at the latest time of \(\varepsilon^2/\beta\).
\end{lem}
\begin{proof}
    Consider \(f(\bar{x},\bar{y},t) = y_1^2 + y_2^2 - (1-\beta)(x_1^2 + x_2^2) + \beta t\) for some \(0\leq \beta \leq 1\), then
    \begin{align*}
        \left(\frac{d}{dt} - \Delta_{M_t}\right)f(\bar{x},\bar{y},t)
        &= \left(\frac{d}{dt} - \Delta_{M_t}\right)\left(|(\bar{x},\bar{y})|^2 - x_1^2 - x_2^2 - (1-\beta)(x_1^2 + x_2^2) + \beta t\right)\\
        &= -6 + 2(2-\beta)(|\nabla^{M_t}x_1|^2 + |\nabla^{M_t}x_2|^2) + \beta
        \leq 2-3\beta,
    \end{align*}
    and this is non-positive when \(\beta \geq 2/3\). The first claim now follows from the weak maximum principle. For the second claim, note that since the initial surface is diffeomorphic to \(\Sb^3\) and topological changes can only occur at singular times, it follows that \(M_t\cong \Sb^3\) up to the first singular time. If no singularity has formed before time \(\varepsilon^2/\beta\), then \(M_t\) must still cross \(P_2\) at this time, but is also forced by the barrier to lie inside a cone with vertex at the origin, which is not possible if \(M_t\cong \Sb^3\). Therefore a singularity has to have formed before this time. 
\end{proof}

For the inner barrier we use a shrinking sphere. The explicit inner and outer barriers that guarantee a non-vanishing flow are given by the following proposition.  

\begin{prop}\label{prop:non_vanishing}
    Given \(\varepsilon ,\rho > 0\), \(\beta < 1\), \(p > \rho\), define \(g,h: \R_{\geq 0}\to \R_{\geq 0}\) by 
    \begin{align*}
        g(x) \coloneqq \sqrt{\rho^2 - (x-p)^2} 
        \quad \textnormal{and} \quad
        h(x) \coloneqq \sqrt{(1-\beta)x^2 + \varepsilon^2}.
    \end{align*} 
    Fix \(2/3\beta<1\) and suppose that \(M_0 = \Sigma_{u}\) satisfies hypothesis (H2) and \(u\) is such that \(g < u < h\). If \(\rho\) and \(\varepsilon\) are such that \(\rho^2 > 6\varepsilon^2/\beta\), then the mean curvature flow \(\{M_t\}_{t>0}\) does not vanish at the first singular time. 
\end{prop}
\begin{proof}
    First note that the condition \(g<u<h\) implies that upon rotation
    \begin{align*}
        B_\rho((p\cos\psi,p\sin\psi,0,0))
        \subset M_0
        \subset \{(\bar{x},\bar{y})\in \R^4:(1-\beta)(x_1^2 + x_2^2) \geq y_1^2 + y_2^2 - \varepsilon^2\},
    \end{align*}
    for any \(\psi \in [0,2\pi)\). Therefore, \(M_0\) satisfies the assumptions of Lemma \ref{lem:hyperboloid_comparison}, and a singularity forms at the latest time of \(\varepsilon^2/\beta\). Now, take any \(\rho^2 > {6\varepsilon^2}/{\beta}\) and suppose that \(M_t\) has vanished at time \({\varepsilon^2}/{\beta}\). Since \(B_\rho((p\cos\psi,p\sin\psi,0,0))\subset M_0\), it follows from the avoidance principle that \(B_{\sqrt{\rho^2-6\varepsilon^2/\beta}}((p\cos\psi,p\sin\psi,0,0)) \subset M_{\varepsilon^2/\beta}\),
    and the radius of the ball on the left hand side is strictly positive by our choice of \(\rho\). However, by assumption, \(M_{\varepsilon^2/\beta}\) has vanished, and therefore cannot contain a 4-ball in its interior. So we have a contradiction.
\end{proof}

Next, we show using the Sturmian theorem detailed in Section \ref{sec:so_2_so_2_invariant_flow}, that if a profile curve crossing \(P_2\) starts with exactly one global maximum and one local minimum at \(x=0\) and no other critical points, then this structure is preserved under the mean curvature flow. 

\begin{prop}\label{prop:singularity_at_origin}
    Suppose that \(M_0 = \Sigma_u\) is mean convex and that there exists \(0<p<\xi_2\), such that \(u_{x}(\cdot,0)>0\) on \((0,p)\) and  \(u_{x}(\cdot,0)<0\) on \((p,\xi_2]\), and \(u_x(0,0) = u_x(p,0) = 0\). Then there exists \(\varepsilon_0\) such that if \(M_0\) satisfies the assumptions of Proposition \ref{prop:non_vanishing} with \(\varepsilon<\varepsilon_0\), then for all \(t<t_0\), there exists \(p(t)\) such that \(u_{x}(\cdot,t)>0\) on \((0,p(t))\) and \(u_x(\cdot,0) < 0\) on \((p(t),\xi_2]\).
\end{prop}
\begin{proof}
    Take \(\beta\), \(\rho\), \(\varepsilon\) determined by Proposition \ref{prop:non_vanishing}, and define \(\varrho \coloneqq \sqrt{\rho^2-\varepsilon^2/\beta}\). Now let \(a_1,\ a_2\) be such that \(u(a_i,0) = \varrho/2\) for \(i=1,2\) and let \(b_1, \ b_2\) be such that \(u(b_i,0) = \varrho \) for \(i=1,2\). The mean convexity of the flow then implies that \(u(x,t)<\varrho/2\) for \(x\in [0,a_1)\cup (a_2,\xi_2(t)]\) and \(u(x,t)<\varrho\) for \(x\in [0,b_1)\cup (b_2,\xi_2(t)]\), whenever \(t<t_0\). We now place a sphere of radius \(r_1 = \varrho/4\) at \((b_1,0,0,0)\), and spheres of radii \(r_2 = \sqrt{6\varepsilon^2/\beta}\) at \((b_1,0,\varrho/2,0)\), \((b_2,0,\varrho/2,0)\) and \((a_2,0,0,0)\) which are all contained within \(\Sigma_u\). This can be done up to restricting \(\varepsilon\) so that \(\sqrt{6\varepsilon^2/\beta}<\varrho/2\), and up to stretching the profile curve as in Lemma \ref{lem:mean_convexity} (which preserves mean convexity and the structure of the initial profile curve). Possibly further restricting \(\varepsilon\) so that \(\varepsilon < \sqrt{(\varrho/4)^2 - 6\varepsilon^2/\beta}\), the barrier at \((b_1,0,0,0)\) ensures that \(u(b_1,t)> \varepsilon\) for \(t<t_0\), the barrier at \((a_2,0,0,0)\) ensures that \(\xi_2(t)>a_2\) and the barriers at \((b_i,0,\varrho/2,0)\) for \(i=1,2\) ensure that \(u(b_i,t)>\varrho/2\) for all \(t<t_0\). See Figure \ref{fig:singularity_origin} for an illustration.

    \begin{figure}[h]
        \centering
        
        \begin{subfigure}[h]{\textwidth}
            \includegraphics[width=\linewidth]{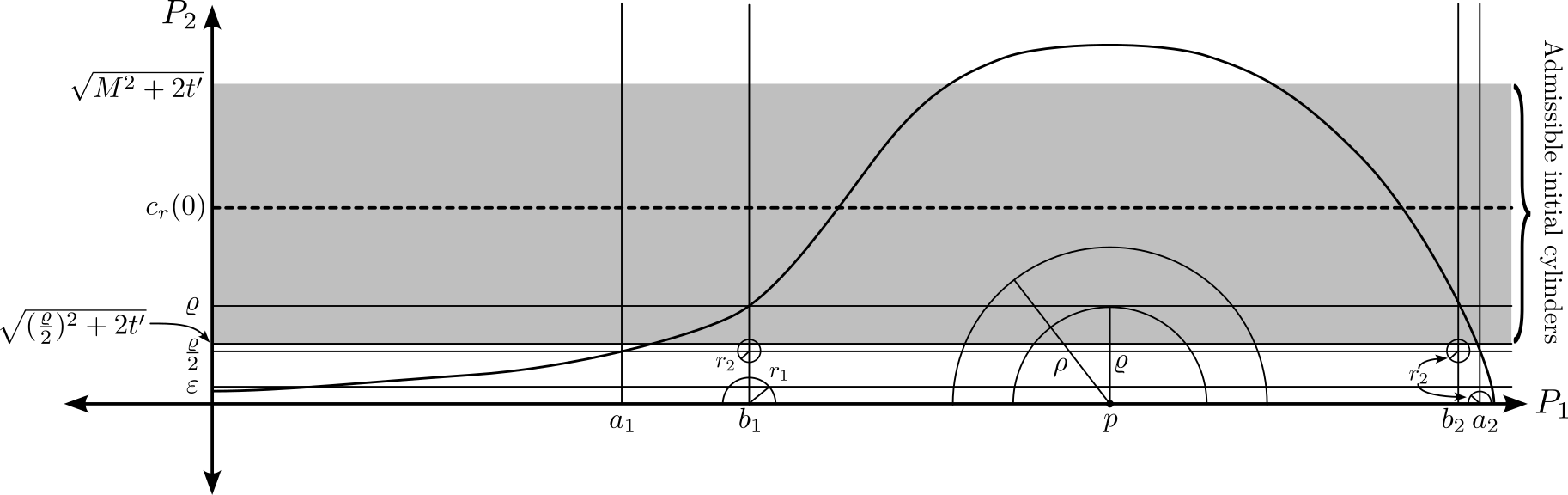}
            \caption{The profile curve with barriers at time \(t=0\)}
            \label{fig:singuarlity_origin_1} 
        \end{subfigure}

        \begin{subfigure}[h]{\textwidth}
            \includegraphics[width=\linewidth]{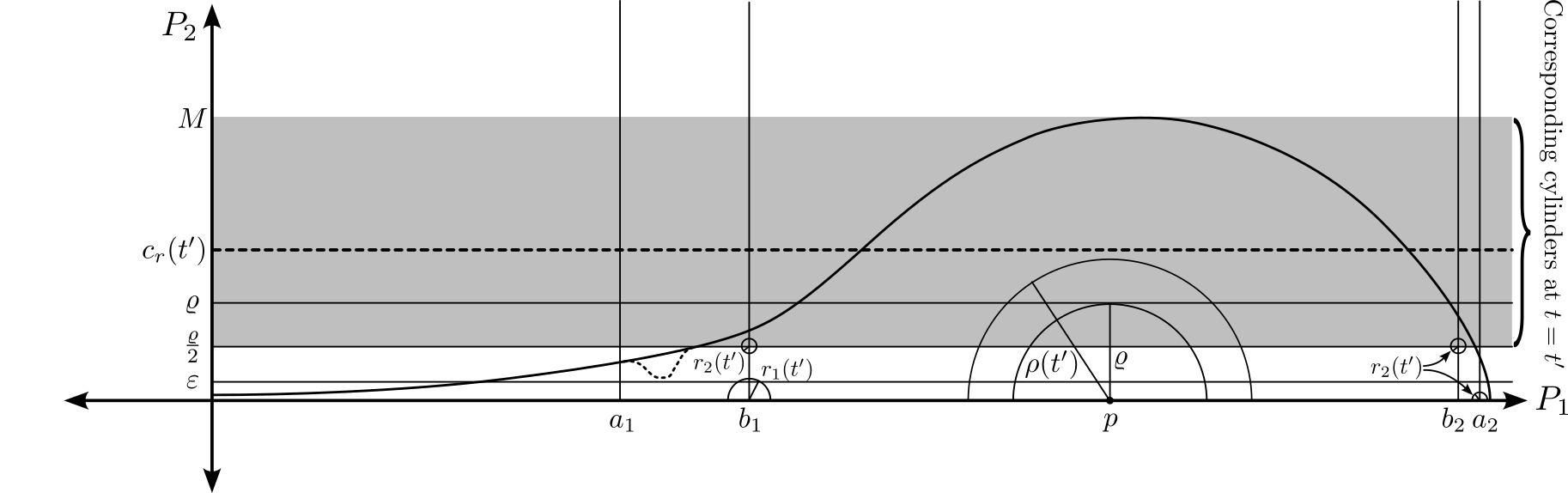}
            \caption{The profile curve with barriers at time \(t=t'\)}
            \label{fig:singularity_origin_2}
        \end{subfigure}
            
        \caption{Application of the Sturmian theorem}
        \label{fig:singularity_origin}
    \end{figure}

    Fix \(t'<t_0\) and define \(c_r(t) = \sqrt{r^2-2(t-t')}\), then \(c_r(t)\) evolves according to \eqref{eq:evolution_equation} and satisfies \(c_r(0) = \sqrt{r^2 + 2t'}\) and \(c_r(t')=r\). For any \(\varrho/2 \leq r \leq M \coloneqq \max u(\cdot,t')\), since \(u(a_i,t)<\varrho/2\) for \(i = 1,2\), it follows that \(u(\cdot,t)-c_r(t)\) will never have a zero at \(a_i\). On the other hand \(u(\cdot,0)-c_r(0)=0\) has at most two zeros for these values of \(r\) by construction, therefore we can apply the Sturmian theorem \ref{thm:sturmian_theorem} to conclude that \(u(\cdot,t) - c_r(t)\) has at most two zeros in \([a_1,a_2]\times[0,t']\). The problem is that \(u\) may take values below \(\varrho/2\) on \([a_1,b_1]\) and \([b_2,a_2]\), as illustrated in Figure \ref{fig:singularity_origin_2}, so we cannot rule out the possibility of a local minimum forming in these regions. However, we chose barriers at \((b_i,0,\varrho/2,0)\) for \(i=1,2\) precisely to ensure that \(u(b_i,t)\geq \varrho/2\) for \(t\leq t'\). Moreover, the maximum value of \(u(\cdot,t)\) is greater than \(\varrho\) and there must therefore be a global maximum in \([b_1,b_2]\), since \(u(\cdot,t)\leq \varrho\) outside of this region. Therefore there must be at least two zeros of \(u(\cdot, t') - c_r(t')\) in \([b_1,b_2]\) for \(\varrho/2<r<M\). By our previous application of the Sturmian theorem, it follows that there must be exactly two zeros, which combined with the structure of the curve means that the only critical point of \(u(\cdot,t')\) can be a global maximum.

    It remains to show that \(u_x(\cdot,t')\) on \([0,b_1]\) and \(u_x(\cdot,t')<0\) on \([b_2,\xi_2(t')]\). We note that it is not possible to apply the Sturmian theorem in these regions, instead we use the strong maximum principle to rule out the existence of other critical points (except for the local minimum at \(x=0\)). Specifically, differentiating the evolution equation \eqref{eq:evolution_equation} with respect to \(x\), we find that \(v = u_x\) satisfies
    \begin{align}\label{eq:gradient_evolution}
        v_t = \frac{1}{(1+u_x^2)}v_{xx} + \frac{1}{x}v_x + \left(\frac{1}{u^2} - 2 \frac{(u_{xx})^2}{(1+u_x^2)^2} - \frac{1}{x^2}\right)v.
    \end{align}
    By assumption, \(v(\cdot,0) >0\) on \((0,b_1]\) and \(v(\cdot,0)<0\) on \((b_2,\xi_2(t)]\). If there exists a first time \(t^*<t'\) at which \(v(x,t^*)\) for some \(x\in (0,b_1]\) or \(x\in [b_2,\xi_2(t))\), and \(U\) a neighbourhood of \(x\) avoiding \(0\) and \(\xi_2(t)\), then the strong maximum principle applied to \eqref{eq:gradient_evolution} implies that \(v\equiv 0\) in \(U\times [0,t^*]\), which contradicts the structure of the profile curve at the initial time.  
\end{proof}

It is an almost immediate corollary of this result that all zeros of the profile curve must belong to a closed set connected to the origin. 

\begin{cor}\label{cor:singularity_connected}
    Suppose that \(M_0= \Sigma_u\) satisfies the assumptions of Proposition \ref{prop:singularity_at_origin}, then the zero set \(Z_{t_0} = \{\xi\in [0,\xi_2(t_0)]:u(\xi,t) = 0\}\) is of the form \([0,\eta]\cup \{\xi_2(t_0)\}\) for some \(\eta < \xi_2(t_0)\).
\end{cor}
\begin{proof}
    By the continuity of \(u(\cdot,t_0)\) (c.f. Lemma \ref{lem:continuity}) it follows that \(Z_{t_0}\) is closed. Moreover, the zero at \(\xi_2(t_0)\) is isolated due to the fact that the gradient of \(u(\cdot,t_0)\) approaches \(-\infty\) near \(\xi_2(t_0)\). Suppose for a contradiction that \(Z_{t_0}\backslash \{\xi(t_0)\}\) has multiple connected components, then between each of then there will be a point at which \(u(\cdot,t_0)\) attains a positive value. Applying the mean value theorem and considering the flow slightly back in time, one finds that the presence of such a point is not compatible with the shape of the profile curve up to the first singular time given by Proposition \ref{prop:singularity_at_origin}.
\end{proof}

We now show that in fact \(Z(t_0) = \{0\}\) by constructing a barrier that gives a uniform lower bound for \(u(\cdot,t_0)\) on any interval not containing the origin.

\begin{prop}\label{prop:isolated_singularity}
    Suppose that \(M_0 = \Sigma_u\) satisfies the assumptions of Proposition \ref{prop:singularity_at_origin}. There exists \(\tau>0\) such that given any compact interval \([c,d]\subset (0,\tau]\) we have \(\inf_{x\in[c,d]}u(x,t_0)>0\). 
\end{prop}

\begin{proof}
    Recall that at the singular time our solution \(M_{t_0}\) contains a spherical barrier of radius \(\varrho = \sqrt{\rho^2 - 6\varepsilon^2/\beta}\). Therefore, there must exist some \(b\) such that \(u(b,t_0)>\varrho\). Take \(\eta\) to be the supremum of the \(Z_{t_0}\), as defined in Corollary \ref{cor:singularity_connected}, then \(u(\cdot,t_0)\) is continuous on \([\eta,b]\) and differentiable on \([\eta,b)\), so we can apply the mean value theorem to conclude that there exists some \(a\) and \(\delta\) such that \(u^*_x(a,t_0) > \delta\). By continuity of the flow up to the first singular time, it follows that there exists \(\varepsilon\) such that \(u_x(a,t)> \delta/2\) for all \(t\in [t_0-\varepsilon, t_0]\). By the structure of the profile curve up to the first singular time given by Proposition \ref{prop:singularity_at_origin}, we must also have \(u_x(x,t)>0\) for \((x,t)\in (0,a)\times [t_0-\varepsilon,t_0)\).

    Recall that the gradient \(v= u_x\) satisfies the evolution equation \eqref{eq:gradient_evolution} 
    \begin{align*}
        v_t = \frac{v_{xx}}{1+v^2} + \frac{v_x}{x} + \left(\frac{1}{u^2} - \frac{2 v_x^2}{(1+ v^2)^2} - \frac{1}{x^2}\right)v 
        \eqqcolon \Lc v.
    \end{align*}
    Now, for any \(0<\sigma <c\), set \(\varphi(x) = \alpha (x-\sigma)^2\). Using that \(\varphi_t = 0\), the expression
    \begin{align*}
        \Lc \varphi - \varphi_t
        &= \frac{\varphi_{xx}}{1+\varphi^2} + \frac{\varphi_x}{x} + \left(\frac{1}{u^2} - \frac{2\varphi_x^2}{(1+\varphi^2)^2} - \frac{1}{x^2}\right)\varphi \\
        &> \left(\frac{\varphi_{xx}}{1+\varphi^2} - \frac{2\varphi_x^2\varphi}{(1+\varphi^2)^2} \right) + \left(\frac{\varphi_x}{x} - \frac{\varphi}{x^2}\right) \\
        & > \frac{2\alpha}{1 + \alpha^2(x-\sigma)^4}\left(1 - 4\alpha^2\tau^4\right)
        + \frac{\alpha(x-\sigma)}{x}\left(1 - \frac{\sigma}{x}\right)
    \end{align*}
    is strictly positive provided \(\alpha < (2\tau^2)^{-1}\). 

    Up to further decreasing \(\alpha\) (which will not affect the positivity of the above expression) we can ensure that \(\varphi(x)< u_x(t_0-\varepsilon)\) for \(x\in [\sigma,\tau]\) and also that \(\varphi(\sigma) < u_x(\sigma,t)\) and \(\varphi(\tau)< u_x(\tau,t)\) for \(t_0-\varepsilon < t<t_0\). We can therefore apply the weak maximum principle to conclude that \(\varphi(x) < u_x(x,t)\) for \((x,t)\in (\sigma,\tau)\times(t_0-\varepsilon, t_0)\). It follows by integrating that
    \begin{align*}
        u(x,t)\geq \frac{\alpha}{3} (x-\sigma)^3
    \end{align*}
    on this domain. Moreover, this bound is strictly positive when we restrict to \(x\in [c,d]\subset [\sigma,\tau]\), and can be extended to \(t_0\), since it is independent of \(t\). 
\end{proof}

It follows from Proposition \ref{prop:isolated_singularity} that any mean curvature flow satisfying the assumptions of Proposition \ref{prop:singularity_at_origin} must form an isolated singularity at the origin. We now show that as a consequence of this, the profile curve \(u(\cdot,t_0)\) is continuously differentiable at 0, which, by Lemma \ref{lem:continuity} implies continuous differentiability along the whole curve away from the end-point \(\xi(t_0)\). We note that to carry out the mean convex approximation in Lemma \ref{lem:mean_convex_approximation} it is crucial that the profile curve is continuously differentiable at zero as opposed to merely differentiable. Showing this requires considering the limit flow along a sequence \((x_i,t_0)\) converging to \((0,t_0)\). The argument shares similarities with  the tangent flow argument in Section \ref{sec:singularity_formation}, so we omit some details.

\begin{lem}\label{lem:differentiability}
    Suppose that \(M_0 = \Sigma_u\) satisfies the assumptions of Proposition \ref{prop:singularity_at_origin}, then \(u(\cdot,t_0)\) has strictly positive derivative in a neighbourhood of \(0\) and is continuously differentiable away from \(\xi_2(t_0)\).
\end{lem}
\begin{rem}
    For the purposes of showing differentiability at 0, it is enough that \(u_x(\cdot,t_0)\geq 0\) an a neighbourhood of 0. However, we require the strict inequality later in this paper.
\end{rem}
\begin{proof}
    By the same argument as in the proof of Proposition \ref{prop:isolated_singularity} we know that there is some point \(b\) at which \(u_x(b,t_0)>0\). Since we know \(u_x(\cdot,t_0)\) is strictly positive away from \(x=0\), the evolution equation for the gradient \eqref{eq:gradient_evolution} is still valid up to the singular time, and in particular on \([a, b] \times(t_0-\varepsilon,t_0]\) for any \(a < b\). If \(u_x(x,t_0)=0\) for some \(x\in (a,b)\), then by the strong maximum principle we would have \(u_{x}(\cdot,t)\equiv 0\) back in time, which contradicts the structure of the profile curve given by Proposition \ref{prop:isolated_singularity}. Therefore, there is a neighbourhood of the origin on which \(u_x(\cdot,t_0) > 0\). Combined with the flatness estimate in Lemma \ref{lem:singularities_are_flat}, this implies that for all \(\varepsilon\) there exists \(\mu_\varepsilon\) such that 
    \begin{align*}
        0 \leq (u^*(x,t_0)-u^*(0,t_0))/x \leq \varepsilon
    \end{align*}
    for \(x<\mu_\varepsilon\), so the the derivative as \(x\to 0^+\) exists and is zero. By symmetry across the \(y_1\)-axis, the limit as \(x\to 0^-\) exists and is equal to zero. Thus, \(u_x(0,t_0) = 0.\) 

    It remains to show continuous differentiability near zero, that is that \(\lim_{x\to 0}u_x(x,t_0) = 0\). By the same reasoning as Proposition \ref{prop:graphicality}, this is equivalent to showing that \(\lim_{p\to 0}\tilde{\nu}(p)\cdot \eb_3 = 1\), where \(\tilde{\nu}(p)\) is the normal to the profile curve at some \(p\in P_0\). Suppose for a contradiction that for all \(\delta > 0\) there exists a sequence \(p_k\to 0\) such that \(\tilde{\nu}(p_k) \cdot \eb_3 < 1-\delta\). By Theorem 1.10 of \cite{Haslh17}, the limit of the parabolic rescaling \((p,t)\mapsto (s_k^{-1}(p-p_k), \ s_k^{-2}(t-t_0))\), where \(s_k=H(p_k,t_0)\), converges locally smoothly as \(k\to\infty\) to a limit flow with convex time slices. We show that the only possible limit flow is a shrinking \(\R^2\times \Sb^1\) with the Euclidean factor parallel to \(P_1\) and the spherical factor parallel to \(P_2\).  On this cylinder we have \(\tilde{\nu}\cdot \eb_3 \equiv 1\) which contradicts the existence of the sequence \(p_k\to 0\) satisfying \(\tilde{\nu}(p_k,t_0)\cdot \eb_3 < 1-\delta\). 

    \begin{figure}[h]
        \centering
        \includegraphics[width = 0.65\textwidth]{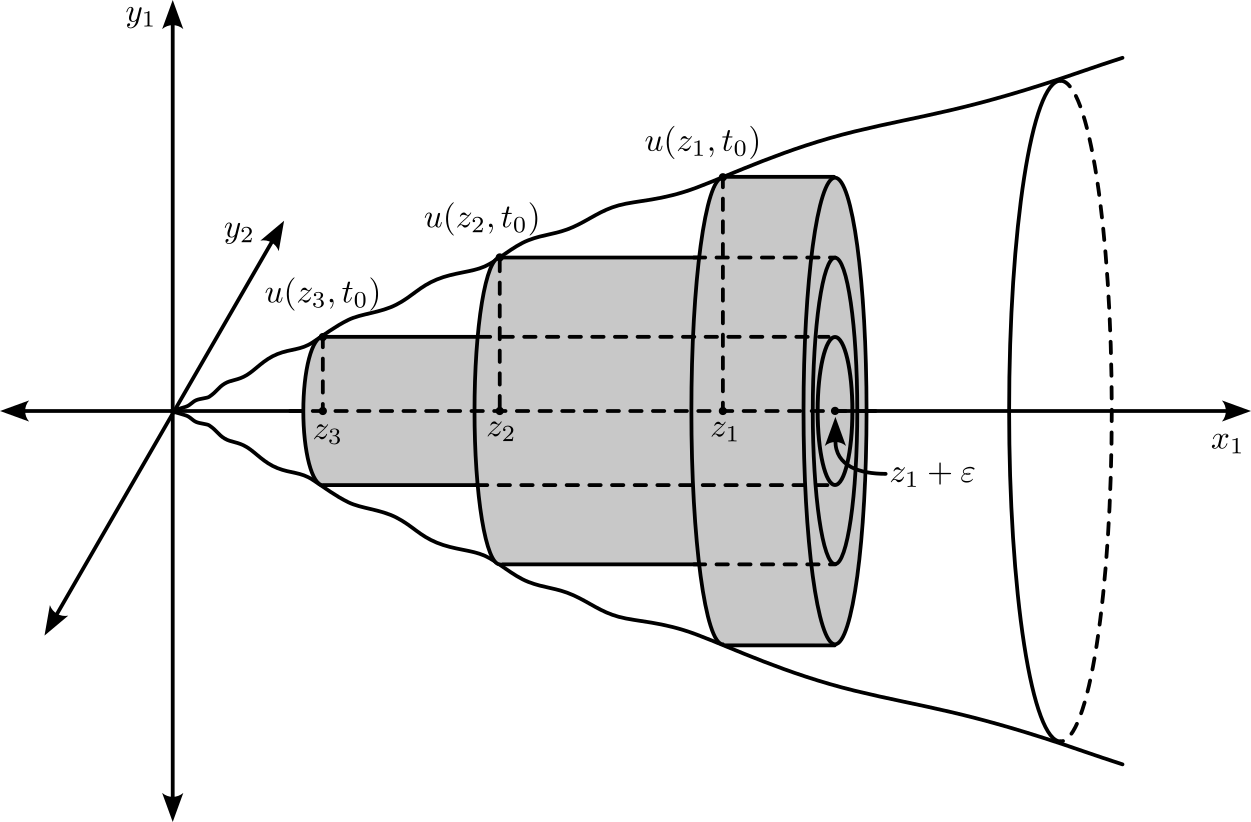}
        \caption{The solid cylinders \(\ell_k\times D_k\)}
        \label{fig:cylinder_containment}
    \end{figure}

    We now demonstrate why the limit flow must be a shrinking \(\R^2\times \Sb^1\). First note that, without loss of generality, each \(p_k\) can be written as \((z_k,0,u(z_k,t_0),0)\) for \(z_k\downarrow 0\). Moreover, by the monotonicity of the profile curve near the origin, for each \(k\in\{1,2,\dots\}\), \(M_{t_0}\) contains the solid cylinder \(\ell_k \times D_k\), where \(\ell_{k} = \{(\bar{x},\bar{y})\in P_1:z_k\leq x_1\leq z_1+ \varepsilon, \ x_2 = 0\}\) for \(\varepsilon>0\) small, and \(D_k\) is a disc of radius \(u(z_k,t_0)\) centred at \((z_k,0,0,0)\) and parallel to \(P_2\), see Figure \ref{fig:cylinder_containment}. By similar reasoning to the proof of Lemma \ref{lem:disc_containment}, the limit flow will either contain a solid cylinder of the form \(\ell\times D\), where \(\ell =\{(\bar{x},\bar{y})\in P_1:x_1\geq 0, x_2 = 0\}\) and \(D\) is some disc in \(P_2\), or it will contain \(H = \{(\bar{x},\bar{y})\in \R^4:x_1\geq 0, \ x_2 = 0, y_2 \leq 0\}\). In the latter case, since \(H\) is a cone, the blow-down of the limit flow will still contain \(H\), but must also be a self-similar shrinker by Theorem 1.11 of \cite{Haslh17}. This forces the limit flow to be a plane, which contradicts that we have normalised the mean curvature to be one along the sequence \(p_k\) and so it cannot vanish identically in the limit. It must therefore follow that the limit flow must contain \(\ell \times D\), and moreover, the intersection of the limit flow with \(P_2\) is precisely \(\pa D\) (by definition of \(D_k\)), which gives the spherical factor. On the other hand, we now now that the disc \(\tilde{D}\) of radius \(\xi(t_0)\), centred at the origin and parallel to \(P_1\) will not escape to infinity under parabolic rescaling. In fact, by similar reasoning to the proof of Lemma \ref{lem:disc_containment}, it must converge to a copy of \(\R^2\) parallel to \(P_1\). This completes the proof. 
\end{proof}

We now apply the above results to construct an explicit profile curve that is mean convex, non-vanishing and forms an isolated singularity at the origin at the first singular time, thus demonstrating the validity of statement (1) of Theorem \ref{thm:main_theorem}. Note that the construction below is just one of many possibilities. 

\begin{eg}\label{eg:profile_curve}
    Consider \(u\in C^0([0,p+1])\cap C^\infty([0,p+1))\) given by
    \begin{align}\label{eq:profile_curve}
        u(x) = 
        \begin{cases}
            \alpha x^2 + \mu & \text{ for } x\in [0,p-1] \\
            \psi(x) & \text{ for } x\in \left(p-1,\ p-\frac{1}{2}\right) \\
            \sqrt{1-(x-p)^2} & \text{ for } x\in \left[p-\frac{1}{2},\ p+1\right],
        \end{cases}
    \end{align}
    with \(\alpha,\ \mu\) small enough that \(u(p-1)<u(p)\), and \(\psi\) a smooth, strictly increasing interpolation function. The definition of \(u\) ensures that \(u_x>0\) for \(x\in (0,p)\) and \(u_x, \ u_{xx} <0\) for \(x\in (p,\xi_2)\) independent of its parameters; this structure is needed to apply Lemma \ref{lem:mean_convexity} and Proposition \ref{prop:singularity_at_origin}.

    \begin{figure}[h]
        \centering
        \includegraphics[width = 0.75\textwidth]{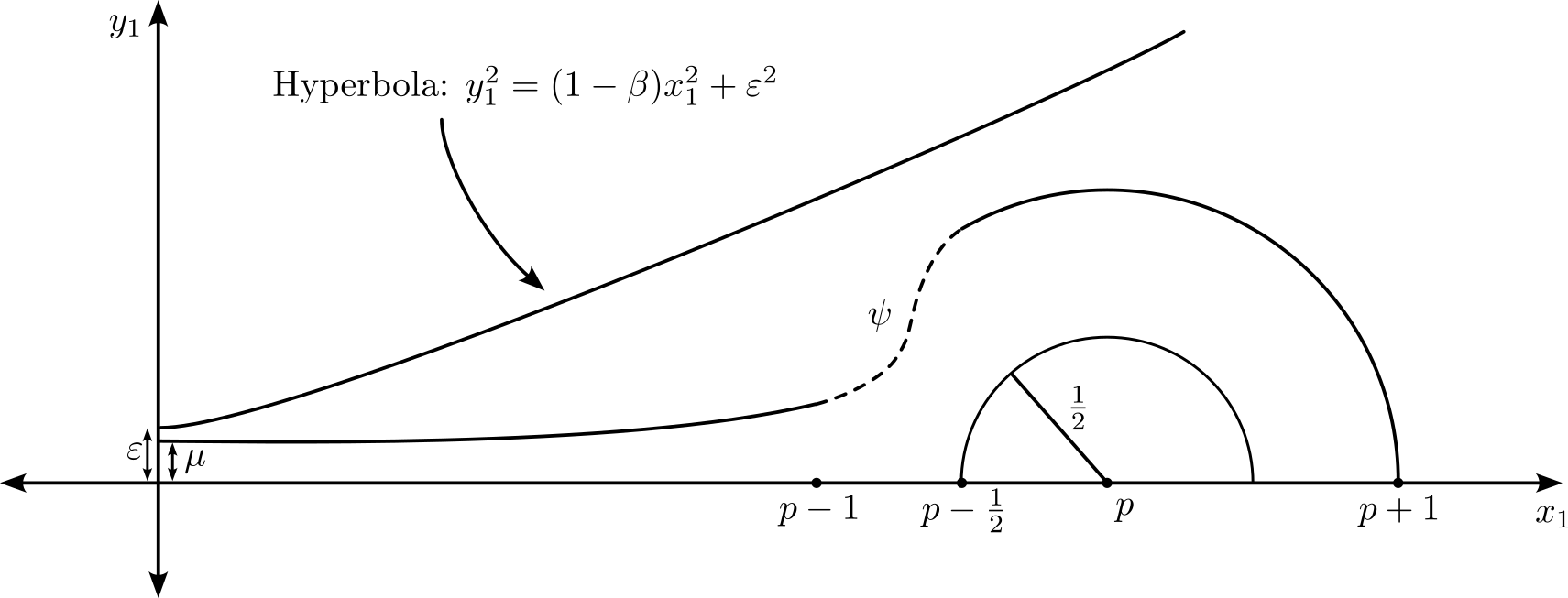}
        \caption{The profile curve \(u\) constructed in equation \eqref{eq:profile_curve}}
        \label{fig:profile_curve}
    \end{figure}

    By construction, the graph of \(u\) bounds the graph of \(g(x) = \sqrt{({1}/{2})^2 - (x-p)^2}\) from above, and this is independent of the choice of \(p,\ \alpha\) and \(\mu\). Now fix some \(2/3\leq \beta < 1\), and take any \(\varepsilon < \varepsilon_0\), where \(\varepsilon_0\) is determined by Proposition \ref{prop:singularity_at_origin}. We now adjust the parameters of \(u\) to ensure that the graph of \(h(x) = \sqrt{(1-\beta)x^2 + \varepsilon^2}\) lies above the graph of \(u\) for this choice of \(\beta\) and \(\varepsilon\). First pick \(p\) large enough that \(h(p-1) > 1\), then by the monotonicity of \(h\) and the fact that 1 is a global maximum for \(u\), 
    \begin{align}\label{eq:hyperbola_bound_1}
        u(x) \leq 1 < h(p-1) < h(x) \quad \textnormal{ for } p-1 < x \leq p.
    \end{align}
    Having made a choice of \(p\), pick \(\alpha\) and \(\mu\) small enough that \(\alpha(p-1)^2 + \mu < \varepsilon\), then 
    \begin{align}\label{eq:hyperbola_bound_2}
        u(x) \leq \alpha(p-1)^2 + \mu < \varepsilon = h(0) < h(x) \quad \textnormal{ for } 0\leq x\leq p-1.
    \end{align}
    Combining \eqref{eq:hyperbola_bound_1} and \eqref{eq:hyperbola_bound_2}, our choices of \(p,\mu\) and \(\alpha\) ensure that \(u(x) < h(x)\) for all \(0\leq x \leq p\), so we satisfy the assumptions of Proposition \ref{prop:non_vanishing}, which ensures that the flow is non-vanishing at the first singular time.

    Next, note that \(u\) satisfies the conditions of Lemma \ref{lem:mean_convexity}, so if \(\Sigma_u\) is not already mean convex, we can apply Lemma \ref{lem:mean_convexity} to find some \(\lambda>1\) such that the stretched hypersurface \(\Sigma_{u_\lambda}\) is mean convex. Moreover, if the graph of \(u\) lies below a hyperbola, then so does the graph of the rescaled function \(u_\lambda\). On the other hand, the graph of \(u_\lambda\) bounds the half circle of radius \(\frac{1}{2}\) and centre \((\lambda p, 0,0,0)\). Therefore our profile curve still satisfies the assumptions of Proposition \ref{prop:non_vanishing}, and additionally satisfies the mean convexity assumption of Proposition \ref{prop:singularity_at_origin}, so at the first singular time an isolated singularity forms at the origin. 
\end{eg}

%% file: 5_instant_smoothness.tex
In this section, we prove an instant smoothness result after the first singular time for some choice of curve \(u\) constructed according to Example \ref{eg:profile_curve}. To simplify notation, we set \({u}^*(x) = u(x,t_0)\), and take \(\Sigma_{{u}^*}\) as our new initial datum. Due to the analysis of Section \ref{sec:topological_change_example}, we know \(\Sigma_{{u}^*}\) is a smooth hypersurface away from the origin. We want to show that \(F_t(\Sigma_{u^*})\) instantly becomes smooth under the level set flow. To prove this, we approximate \(\Sigma_{u^*}\) by a non-decreasing sequence of smooth, mean convex hypersurfaces \(\Sigma_{u^j}\) and show that there exists \(T\) such that the second fundamental form of \(F_t(\Sigma_{u^j})\) is uniformly bounded on any time interval \([\delta,T]\) with \(\delta>0\). We then conclude using Remark \ref{rem:mean_convex_limit}. 

The fact that we need to take our approximating surfaces to be mean convex is subtle but crucial, and involves some slightly delicate analysis. We first state and prove the following bound on the second derivative of \(u\). This estimate will play a key role in our construction, as it provides sufficiently strong control for the positive terms in the expression for mean curvature to dominate the potentially negative contribution arising from the second derivative.

\begin{lem}\label{lem:u_xx_bound}
    Let \(M_0 = \Sigma_u\) be mean convex and satisfy hypothesis (H1) or (H2). For all \(0\leq t<t_0\) there exists some \(C>0\) such that
    \begin{align}\label{eq:u_xx_bound}
        \frac{u_{xx}}{1+u_x^2} \leq \frac{C}{1+C}\left(\frac{1}{u} - \frac{u_x}{x}\right).
    \end{align}
\end{lem}
\begin{proof}
    For any mean convex mean curvature flow, it was shown by Huisken and Sinestrari that there exists uniform \(C=C(M_0)>0\) such that \(|A|^2/H^2< C^2\) on \(M_t\) for \(0\leq t< t_0\), where \(A\) denotes the second fundamental form \cite{Huisk99}. Using the expression for the principal curvatures given by \eqref{eq:principal_curvatures}, one finds that 
    \begin{align*}
        |A|^2 = \frac{u_{xx}^2}{(1+u_x^2)^3} + \frac{1}{u^2(1+u_x^2)} + \frac{u_x^2}{x^2(1+u_x^2)},
    \end{align*}
    and so
    \begin{align*}
        \frac{u_{xx}^2}{(1+u_x^2)^3} 
        \leq |A|^2 
        \leq C^2 H^2
        = C^2 \left(-\frac{u_{xx}}{(1+u_x^2)^{\frac{3}{2}}} + \frac{1}{u(1+u_x^2)^{\frac{1}{2}}} -\frac{u_x}{x(1+u_x^2)^\frac{1}{2}} \right)^2.
    \end{align*}
    Multiplying by \(1+u_x^2\) and taking a square root (using that \(H>0\) by mean convexity) we obtain
    \begin{align*}
        \frac{u_{xx}}{1+u_x^2} 
        \leq C \left(- \frac{u_{xx}}{1+u_x^2} + \frac{1}{u} - \frac{u_x}{x}\right) 
    \end{align*}
    Rearranging gives the result.
\end{proof}

\begin{lem}\label{lem:mean_convex_approximation}
    Let \(\Sigma_{u^*}\) be as before. Then \(\Sigma_{u^*}\) admits an approximation by smooth, mean convex hypersurfaces \(\Sigma_{u^j}\) represented by functions \(u^j\in C^0([\delta_02^{-(j+1)},\xi_2(t_0)])\cap C^\infty((\delta_0 2^{-(j+1)},\xi_2(t_0)))\) such that \(u^{j+1}\leq u^j \leq u^*\)for sufficiently small \(\delta_0>0\). Specifically, the approximation satisfies
    \begin{equation}\label{eq:approximation_property}
        u^j(x)
        \begin{cases}
            \leq {u}^*(x) & \textnormal{ if } {2^{-(j+1)}}\delta_0\leq x \leq 2^{-j}\delta_0 \\[2pt]
            = {u}^*(x) & \textnormal{ if } x > {2^{-j}}{\delta_0}. 
        \end{cases}
    \end{equation}
    Moreover, \(u_x^j>0\) whenever \(u^j \neq u^*\).
\end{lem}
\begin{proof}
    By Lemma \ref{lem:differentiability}, \(u^*\) is continuously differentiable and monotonic increasing in a neighbourhood of \(0\) and \(u^*(0) = u_x^*(0) = 0\). Therefore given \(\varepsilon>0\) to be determined later, there exists \(\delta_0\) such that
    \begin{align}\label{eq:epsilon_bounds}
        0 \leq u_x^*(x) < \varepsilon 
        && \textnormal{and} && 
        u^*(x) < \varepsilon x
    \end{align}
    for \(x\in[0,\delta_0]\). 
    
    To construct our approximation, we interpolate between \(u^*\) and an appropriately chosen function \(p^j\), defined on \([2^{-(j+1)}\delta_0,2^{-j}\delta_0]\). We require that \(\Sigma_{p^j}\) is smooth and mean convex, and also that \(p^j(2^{-(j+1)}\delta_0) = 0\) and \(p^j\leq m_j\) where \(m_j\) is the minimum value achieved by \(u^*\) on \([2^{-(j+1)}\delta_0,2^{-j}\delta_0]\). Due to the monotonicity of \(u^*\), this value is always achieved at \(2^{-(j+1)}\delta_0\), and so \(m^j \leq \varepsilon 2^{-(j+1)}\delta_0\). We claim that 
    \begin{align*}
        p^j(x) = m_j(2^{j+1}\delta_0^{-1}x -1)^{1/2}
    \end{align*}
    has these properties; all except mean convexity are immediate. To show mean convexity, we calculate
    \begin{align*}
        p_x^j = m_j 2^j\delta_0^{-1}(2^{j+1}\delta_0^{-1}x-1)^{-1/2} 
        && \textnormal{and} &&
        p_{xx}^j(x) = -m_j (2^j\delta_0^{-1})^2(2^{j+1}\delta_0^{-1}x-1)^{-3/2}.
    \end{align*} 
    Since \(p_{xx}^j\) is negative, it is enough to show that \(\frac{1}{p^j} - \frac{p_x^j}{x} > 0\). Using that \(m_j\leq \varepsilon 2^{-(j+1)}\delta_0\) and \(x\geq 2^{-(j+1)}\delta_0\), we obtain the bounds
    \begin{align*}
        p^j(x) \leq \varepsilon 2^{-(j+1)}\delta_0(2^{j+1}\delta_0^{-1}x - 1)^{1/2} 
        && \textnormal{and} &&
        \frac{p^j_x(x)}{x} \leq \varepsilon 2^j\delta_0^{-1}(2^{j+1}\delta_0^{-1}x-1)^{-1/2}.
    \end{align*}
    It follows that
    \begin{align*}
        \frac{1}{p^j(x)} - \frac{p_x^j(x)}{x} 
        > 2^j\delta_0^{-1}(2\varepsilon^{-1}- \varepsilon)(2^{j+1}\delta_0^{-1}x-1)^{-1/2} >0,
    \end{align*}
    provided \(\varepsilon\) is chosen small enough that \(2\varepsilon^{-1}-\varepsilon>0\).

    \begin{figure}[h]
        \centering
        \includegraphics[width = 0.85\textwidth]{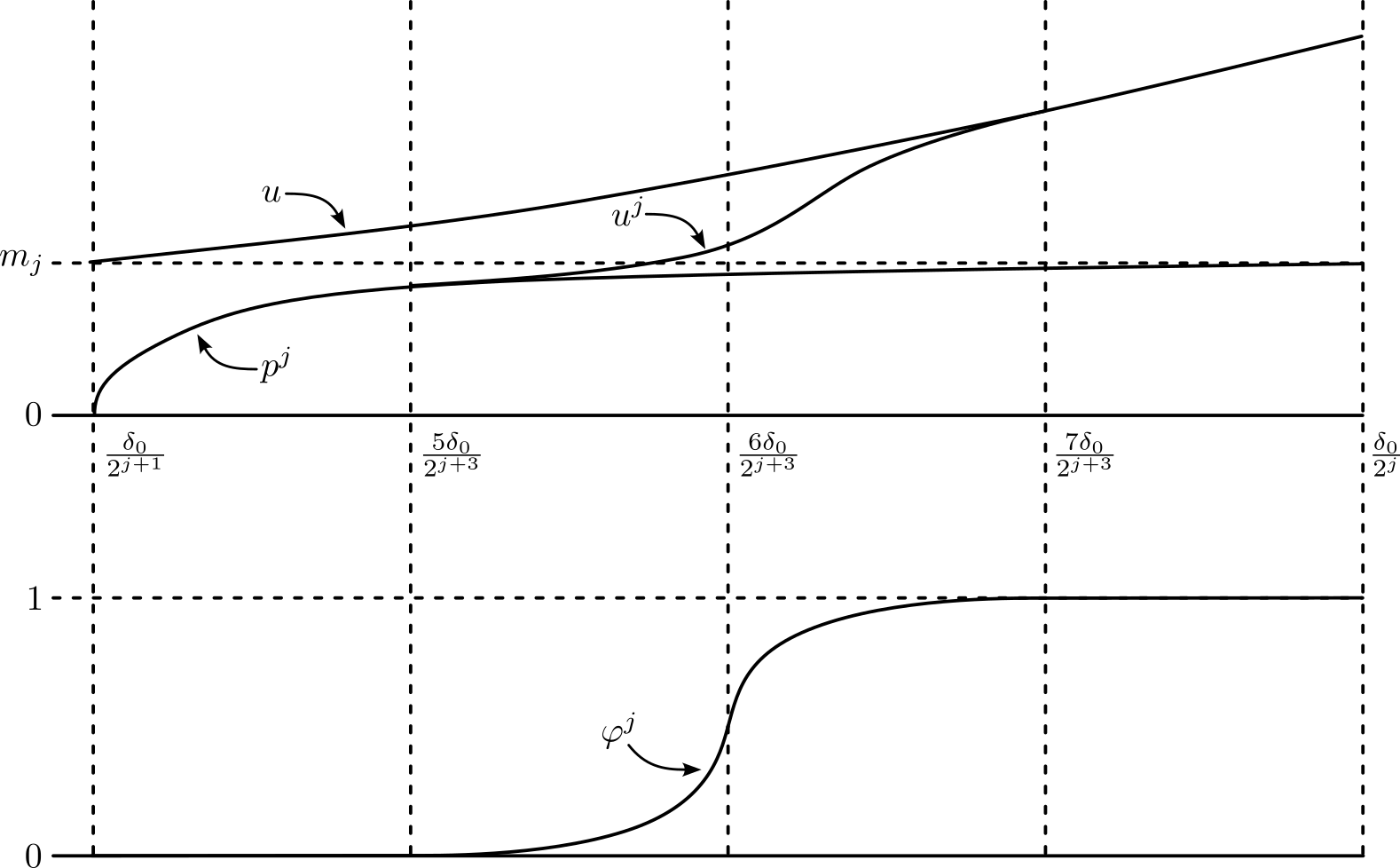}
        \caption{The interpolation function \(u^j\)}
        \label{fig:interpolation}
    \end{figure}

    Now consider any smooth, monotonic increasing function \(\varphi^j\) on \([2^{-(j+1)}\delta_0, \ 2^{-j}\delta_0]\) such that \(\varphi^j \equiv 0\) on \([2^{-(j+1)}\delta_0, \ 5\cdot 2^{-(j+3)}\delta_0]\), \(\varphi^j \equiv 1\) on \([7\cdot 2^{-(j+3)}\delta_0, \ 2^{-j}\delta_0]\), and \(0\leq \varphi^j \leq 1\), \(\varphi_x^j>0\) on \((5\cdot 2^{-(j+3)}\delta_0, \ 7\cdot 2^{-(j+3)}\delta_0)\). Additionally, \(\varphi^j\) should satisfy 
    \begin{align*}
        \sup_{x\in[5\cdot 2^{-(j+3)}\delta_0, 7\cdot 2^{-(j+3)}\delta_0]} \varphi_x^j(x) \leq C_1 2^{j+3}\delta_0^{-1}
        &&
        \sup_{x\in[5\cdot 2^{-(j+3)}\delta_0, 7\cdot 2^{-(j+3)}\delta_0]} |\varphi_{xx}^j(x)| \leq C_2 (2^{j+3}\delta_0^{-1})^2
    \end{align*}
    for some \(C_1, C_2 > 0\) independent of \(j\). there are many standard constructions of such a function. Define \(u^j\coloneqq \varphi^ju^* + (1-\varphi^j)p^j\) on \(x\in [2^{-(j+1)}\delta_0,2^{-j}\delta_0]\). Then \(u^j = p^j\) on \([2^{-(j+1)}\delta_0,\ 5\cdot 2^{-(j+3)}\delta_0]\) and \(u^j = u^*\) on \([7\cdot 2^{-(j+3)}\delta_0,\ 2^{-j}\delta_0]\). Thus we can set \(u^j = u^*\) on \([2^{-j}\delta_0, \ \xi_2]\) to ensure that \(u^j\) satisfies \eqref{eq:approximation_property}. See Figure \ref{fig:interpolation} for an illustration.
    
    It remains to show the mean convexity of \(u^j\) in the region \((5\cdot 2^{-(j+3)}\delta_0,\ 7\cdot 2^{-(j+3)}\delta_0)\). It suffices to show that 
    \begin{align*}
        -\frac{u_{xx}^j}{1+(u_x^j)^2} + \frac{1}{u^j} - \frac{u_x^j}{x} > 0.
    \end{align*}
    We start by finding lower bounds for each term individually. First, \(u^j \leq u^*\) by construction, so 
    \begin{align}\label{eq:second_term_bound}
        \frac{1}{u^j} \geq \frac{1}{u^*}.
    \end{align}
    Next, note that \(p_x^j\) achieves a maximum value of \(\varepsilon\) at \(5\cdot 2^{-(j+3)}\delta_0\), so
    \begin{align*}
        u_x^j = \varphi_x^j(u^*-p^j) + \varphi^ju_x^* + (1-\varphi^j)p_x^j
        \leq \varphi_x^ju^* + u_x^* + p_x^j \leq \varepsilon(8C_1 + 2)
    \end{align*} 
    where we have also used that \(u_x^*(x)\leq \varepsilon\) and \(u^*(x)\leq u^*(2^{-j}\delta_0)\leq \varepsilon 2^{-j}\delta_0\) (see \eqref{eq:epsilon_bounds}). Combined with the fact that \(x> 2^{-(j+1)}\delta_0\), it follows that
    \begin{align}\label{eq:third_term_bound}
        -\frac{u_x^j}{x}\geq -\varepsilon2^j\delta_0^{-1}C_1'
    \end{align}
    for \(C_1' = (8C_1+2)\). Finally, using that \(|p_{xx}^j|\) achieves a maximum value of \(4\varepsilon2^j\delta_0^{-1}\) at \(5\cdot 2^{-(j+3)}\delta_0\) as well as the upper bounds on \(u^*\) and \(u^*_x\) given by \eqref{eq:epsilon_bounds}, we have
    \begin{align*}
        u_{xx}^j 
        &= \varphi_{xx}^j(u^*-p^j) + 2\varphi_x^j(u_x^* - p_x^j) + \varphi^ju_{xx}^* + (1-\varphi^j)p_{xx}^j \\
        &\leq |\varphi_{xx}^j|(|u^*| + |p^j|) + 2|\varphi_{x}|(|u_x^*| + |p^j_x|) + (1-\varphi^j)|p_{xx}^j| + \varphi^ju_{xx}^* \\
        &\leq \varepsilon2^j\delta_0^{-1}(128 C_2 + 32C_1 + 4) + \varphi^ju_{xx}^*.
    \end{align*}
    It follows from Lemma \ref{lem:u_xx_bound} (which can be extended up to the singular time away from the singular set) that there exists \(C>0\) such that
    \begin{align*}
        u_{xx}^* \leq C(1+C)^{-1}(1+\varepsilon^2)\frac{1}{u^*},
    \end{align*}
    where we have also used that \(u_x^*\geq 0\) in the region we are considering. Therefore, 
    \begin{align*}
        u_{xx}^j \leq \varepsilon2^j\delta_0^{-1}C_2' + C(1+C)^{-1}(1+\varepsilon^2) \frac{1}{u^*},
    \end{align*}
    where \(C_2' = 128C_2 + 32 C_1 + 4\). The right hand side of this inequality is positive, so 
    \begin{align}\label{eq:first_term_bound}
        -\frac{u_{xx}^j}{1 + (u_x^j)^2} 
        \geq -\varepsilon2^j\delta_0^{-1}C_2' - C(1+C)^{-1}(1+\varepsilon^2)\frac{1}{u^*}.
    \end{align}
    Combining inequalities \eqref{eq:second_term_bound}, \eqref{eq:third_term_bound}, \eqref{eq:first_term_bound} and using that \(1/u^* \geq \varepsilon^{-1}2^j\delta_0^{-1}\), we have
    \begin{align*}
        -\frac{u_{xx}^j}{1+(u_x^j)^2} + \frac{1}{u^j} - \frac{u_x^j}{x} 
        &\geq \frac{1}{u^*}(1-C(1+C)^{-1}(1+\varepsilon^2)) - \varepsilon 2^j\delta_0^{-1}(C_1' + C_2')\\
        &\geq 2^j\delta_0^{-1}(\varepsilon^{-1}(1+C)^{-1} -\varepsilon(C(1+C)^{-1}+ C_1' + C_2')),
    \end{align*}
    which is positive for \(\varepsilon\) small enough. Note that our choice of \(\varepsilon\) depends only on \(C\), \(C_1\), \(C_2'\) and in particular can be chosen independently of \(j\). 
\end{proof}

To complete the proof of instant smoothness for \(\Sigma_{u^*}\), it remains to establish uniform derivative bounds for \(F_t(\Sigma_{u^j})\) on \([\delta,T)\) for any \(\delta>0\). Away from the singularity, we are able to apply a pseudolocality result. Close to the singularity we make use of the Evans-Spruck gradient estimates for graphical mean curvature flow. We recall both these theorems below, with brief notes on how they will be applied.

Loosely speaking, pseudolocality asserts that if the initial datum of a mean curvature flow is sufficiently regular in a given region, then one can obtain quantitative curvature bounds for a short time in a slightly smaller region. The exact statement is as follows.

\begin{prop}\label{prop:pseudolocality}
    There exist positive constants \(\beta_0\in (0,1)\) and \(C\) such that the following holds. Suppose \(\{M_t\}_{t\in[0,T]}\) is a regular mean curvature flow in \(B_{4r}(p)\), moreover assume that the initial surface \(M_0\) can be expressed as the graph of a function \(u\) over a plane. If \(\|u\|_{C^4}\leq \beta_0r^{-3}\), then \(|A(x,t)|\leq Cr^{-1}\) for all \(0\leq t\leq \min\{r^2\beta_0,T\}\) and \(x\in M_t\cap B_r(p)\).
\end{prop}
\begin{proof}
    The result follows from rescaling Proposition 4.3 of \cite{Brend16b}. This result is only stated for mean curvature flow in \(\R^3\), but the proof extends to arbitrary dimension without issue.
\end{proof}

Since we require curvature bounds that are independent of \(j\), we can only apply this result a fixed positive distance away from the origin where, by construction, the approximation \(\Sigma_{u^j}\) coincides with \(\Sigma_{u^*}\) for all \(j\) (see \eqref{eq:approximation_property}).

Next, the Evans-Spruck gradient estimates are estimates on the \(k\)\textsuperscript{th} derivatives of a smooth solution \(v\) to the full graph equation
\begin{align}\label{eq:full_graph_equation}
    v_t = \left(\delta_{ij}-\frac{v_{x_i}v_{x_j}}{1+|Dv|^2}\right)v_{x_ix_j}.
\end{align}
The estimates are as follows.

\begin{prop}[\cite{Evans92b}, {\cite[Corollary 2.3]{Altsh95}}]\label{prop:evans_spruck}
    For \(s_1 < s_2\), \(\rho>0\) and \(q\in \R^n\), set 
    \begin{align*}
        \Omega = B_\rho(q)\times (s_1,s_2).
    \end{align*}
    Suppose that \(v\in C^3(\Omega)\) solves the full graph equation in \(\Omega\) with \(M=\sup_\Omega|v|<\infty\). For any \(\varepsilon>0\), there is a constant \(C = C(M,\varepsilon,n)\) such that
    \begin{align*}
        |\nabla v|\leq C 
        \quad \textnormal{ on } \quad  
        \Omega_{\varepsilon} = B_{\rho-\varepsilon}(q)\times(s_1+\varepsilon^2,s_2).
    \end{align*}
\end{prop}

Classical interior Hölder and Schauder estimates then give bounds on higher derivatives.

\begin{prop}[{\cite[Lemma 2.4]{Altsh95}}]\label{prop:evans_spruck_schauder}
    Let \(\Omega\) be as before, and suppose that \(v\in C^3(\Omega)\) solves the full graph equation in \(\Omega\) with \(M=\sup_\Omega |v|<\infty\) and \(N = \sup_\Omega|\nabla v|<\infty\). Then for \(\varepsilon>0\) and non-negative \(\ell \in \N\), there exists \(C' = C'(M,N,\varepsilon,n,\ell)\) such that 
    \begin{align*}
        \|v\|_{C^\ell(\bar{\Omega}_\varepsilon)}\leq C'.
    \end{align*}
\end{prop}

The Evans-Spruck estimates apply only when the mean curvature flow can be written as a solution of the full graph equation over a uniform domain for a positive interval of time. To apply them we demonstrate it is possible to describe the portion of \(\Sigma_{u_j}\) close to the origin as a graph over some appropriately chosen subset of \(\{x_1=0\}\). To do so, we first need to show that the functions \(u^j(\cdot,t)\) have a well-defined inverse over a region \([0,\alpha_0]\) that is defined independently of \(t\) and \(j\). To make this more precise, we introduce the concept of an \(\alpha\)-invertible function, which is based on the definition of an \(\alpha\)-domain in \cite[Definition 2.1]{Altsh95}. However, here we choose to frame things in terms of the profile curve rather than the interior of the hypersurface. 

\begin{defn}\label{def:alpha_invertible}
    Suppose \(\Sigma_u\) satisfies hypothesis (H1), and fix \(\alpha>0\). We say that \(u\) is \emph{\(\alpha\)-invertible} if \(u = \rho\) at exactly two points for each \(0<\rho\leq \alpha\) and \(u_x\neq 0\) at these points.
\end{defn}

Crucially, this definition tells that the \emph{only} places where \(u\) can have a value less than \(\alpha\) are near its end-points. Moreover, it implies that there exist \(\delta_1\) and \(\delta_2\) such that the inverses of \(u|_{[\xi_1,\xi_1+\delta_1]}\) and \(u|_{[\xi_2-\delta_2,\xi_2]}\) are well defined functions \(v_1, v_2:[0,\alpha]\to \R_{\geq 0}\), which are smooth on \((0,\alpha]\) by the inverse function theorem. The following result shows that an \(\alpha\)-invertible profile curve remains \(\alpha_0\)-invertible for some positive amount of time and some \(\alpha_0<\alpha\).

\begin{lem}\label{lem:alpha_invertible}
    Suppose \(M_0 = \Sigma_u\) is generated by an \(\alpha\)-invertible function \(u\), and suppose that the corresponding mean curvature flow has first singular time \(t_0\). Then there exists \(t_\alpha\in (0,t_0]\) such that \(u(\cdot,t)\) is \(\alpha(t)\)-invertible for \(0<t<t_\alpha\), where \(\alpha(t) = \sqrt{\alpha^2-2t}\). Moreover, if \(\tilde{u}\geq u\) is another \(\alpha\)-domain satisfying the same assumptions as \(u\) and with first singular time \(\tilde{t}_0\), then \(\tilde{u}(\cdot,t)\) is \(\alpha(t)\)-invertible for \(0<t<t_\alpha\), where \(t_\alpha\) is as before.
\end{lem}
\begin{proof}
    Since \(u> \alpha\) away from its endpoints, for small enough \(\varepsilon>0\) we can place a ball of radius \(\varepsilon\) between \(\Sigma_u\) and \(\Sigma_\alpha\). Now take \(t_\alpha = \min \{\varepsilon^2/6, \ \alpha^2/2\}\). If \(\tilde{u}\geq u\), then the same ball will be contained between \(\Sigma_{\tilde{u}}\) and \(\Sigma_\alpha\), so the following arguments hold in this case as well.
    
    We first show that \(u(\cdot,t)\) is \(\alpha(t)\)-invertible for \(t\leq  \min \{t_0,t_\alpha\}\). Given any \(t' < \min \{t_0,t_\alpha\}\), consider some \(0< \rho \leq \alpha(t')\), and define \(\rho(t) = \sqrt{\rho^2 - 2(t-t')}\), so that \(\rho(0) = \sqrt{\rho^2 + 2t'}\leq \alpha\) and \(\rho(t') = \rho\). Due to the spherical barrier, we must have \(u(\cdot,t') = \rho(t')\) at least twice. On the other hand, since \(u(\cdot,0)\) is equal to \(\rho\) twice for all \(\rho< \alpha\), it follows from the Sturmian theorem \ref{thm:sturmian_theorem} (and Remark \ref{rem:curvilinear} to deal with the end-points), that \(u(\cdot,t')\) is equal to \(\rho(t')\) at most twice. Therefore they are equal at exactly two points, and since there was no decrease in the number of zeros of \(u(\cdot,t) - \rho(t)\), the zeros must all be of order 1, so \(u_x(\cdot,t')>0\) at these points. 

    Finally, we want to show that \(t_\alpha \leq t_0\). Suppose not, then \(\Sigma_{u_t}\) becomes singular before \(u(\cdot,t)\) ceases to be \(\alpha(t)\)-invertible. The spherical barrier means that \(\Sigma_{u_t}\) cannot vanish at its first singular time \(t_0\) (since this time is less than \(t_\alpha \leq \varepsilon^2/6\) by assumption). Additionally, since \(\Sigma_u\) is mean convex, Proposition \ref{prop:bubblesheet} tells us that the singularity can only form when \(u=0\), but not at \(\xi_1(t_0)\) and \(\xi_2(t_0)\). However, for each \(t<t_0\), \(u(t)\) is \(\alpha(t)\)-invertible with \(\alpha(t)\geq \alpha(t_0) > \alpha(t_\alpha) \geq 0\), and so we can't have \(u = 0\) away from an end-point. Thus we must have \(t_0\leq t_\alpha\).
\end{proof}

We are now in a position to establish the following.

\begin{prop}\label{prop:uniform_derivative_bounds}
    Let \(\Sigma_{u^*}\) be as before, and let \(\Sigma_{u^j}\) be the mean convex approximation given by Lemma \ref{lem:mean_convex_approximation}. Then there exists \(T\) such that for any \(\delta \in (0,T)\), the second fundamental form of \(F_t(\Sigma_{u^j})\) and its derivatives are uniformly bounded for \(j\geq 1\) and \(\delta < t < T\).
\end{prop}

\begin{proof}
    It follows from Lemma \ref{lem:differentiability} that \(u^*_x >0\) in a neighbourhood of the origin and from \eqref{eq:crossing_P2_smoothness} that \(u^*_x<0\) in a neighbourhood of \(\xi_2\), so the approximation \(u^j\) is \(\alpha\)-invertible for some \(\alpha\) uniform in \(j\). It then follows from Lemma \ref{lem:alpha_invertible} that there exist \(\alpha_0, \ t_0>0\) such that all the \(u^j\)s are \(\alpha_0\)-invertible for \(0<t< t_0\). Recall that we want to split our analysis into two regions, near to and away from the origin. Near to the origin, we argue using the Evans-Spruck gradient estimates, and away from the origin we argue using pseudolocality. In order to patch these two arguments together, we need to ensure that the domains on which we are applying these two arguments overlap on some section of the profile curve.

    To this end, let \(p\) be such that \(u^*(p) = \alpha_0\). Then for small enough \(\varepsilon_0\), we can place a sphere of radius \(\varepsilon_0\) in the interior of \(\Sigma_u\) and this containment holds with \(\Sigma_{u^j}\) replacing \(\Sigma_{u^*}\), up to discarding finitely many \(\Sigma_{u^j}\)s. Then by the avoidance principle, the sphere of radius \(\varepsilon(t) = \sqrt{\varepsilon_0^2-6t}\) remains in the interior of \(F_t(\Sigma_{u^j})\) for \(0\leq t\leq \min\{t_0,\varepsilon_0^2/6\}\). Let \(p_{\varepsilon_0} = p - \varepsilon_0/2\), then \(\varepsilon(t)\geq {\varepsilon_0}/\sqrt{2}\) for \(0\leq t \leq \min\{t_0,\varepsilon_0^2/12\}\eqqcolon T\), which implies that \(\xi^j_1(t) < p_{\varepsilon_0}\), where \(\xi^j_1(t)\) is the left end-point of \(u^j\) at time \(t\). On the other hand, it follows from mean convexity of the flow that for \(x\in (p_{\varepsilon_0},p)\) and any \(j\geq 1\), we must have \(u^j(x,t)<\alpha_0\) for \(t< T\). Let
    \begin{align*}
        X = \{(\bar{x},\bar{y})\in P_0:0\leq x_1\leq p, \ 0\leq y_1 < \alpha_0\} 
        && \textnormal{and} &&
        Y = \{(\bar{x},\bar{y})\in P_0: x_1 \geq p_{\varepsilon_0}\}
    \end{align*}
    then by the above reasoning some portion of the graph of \(u^j(x,t)\) must lie within \(X\cap Y\). Due to this overlap, to obtain a uniform bound on all derivatives of the second fundamental form of \(F_t(\Sigma_{u^j})\), it suffices to find such a bound for points on the flow lying in \(X\) and \(Y\) separately, and then invoke the \(SO(2)\times SO(2)\)-invariance of \(F_t(\Sigma_{u^j})\).

    \smallskip

    \noindent{\bf Case 1: Close to the origin.} We start by showing that close to the origin portions of \(F_t(\Sigma_{u^j})\) satisfy the full graph equation over a domains independent of \(t\) and \(j\). This relies on a quantitative lower bound on the distance between the evolving surface and the origin. Specifically, the first part of Lemma \ref{lem:hyperboloid_comparison} applied with \(\varepsilon=0\) tells us that if for some \(2/3\leq \beta < 1\), 
    \begin{align}\label{eq:Sigma_u^j_containment}
        \Sigma_{u^j} \subset \{(\bar{x},\bar{y})\in \R^4:(1-\beta)(x_1^2 + x_2^2) \geq y_1^2 + y_2^2\}
    \end{align}
    then the level set flow satisfies
    \begin{align}\label{eq:level_set_flow_containment}
        F_{t}(\Sigma_{u^j}) \subset \{(\bar{x},\bar{y})\in \R^4: (1-\beta)(x_1^2 + x_2^2) \geq y_1^2 + y_2^2 + \beta t\}
    \end{align}
    for all \(t\), so long as the solution stays smooth for this time. Since \(\textnormal{int}(\Sigma_{u^j})\subset \textnormal{int}(\Sigma_{u^*})\), the parameter \(\beta\) can be chosen to depend only on \(\Sigma_{u^*}\), so the lower bound is uniform in \(j\). If necessary, the containment \eqref{eq:Sigma_u^j_containment} can be achieved by further stretching the profile curve constructed in Example \ref{eg:profile_curve} along the \(x_1\)-axis.

    \begin{figure}[h]
        \centering
        \includegraphics[width = 0.8\textwidth]{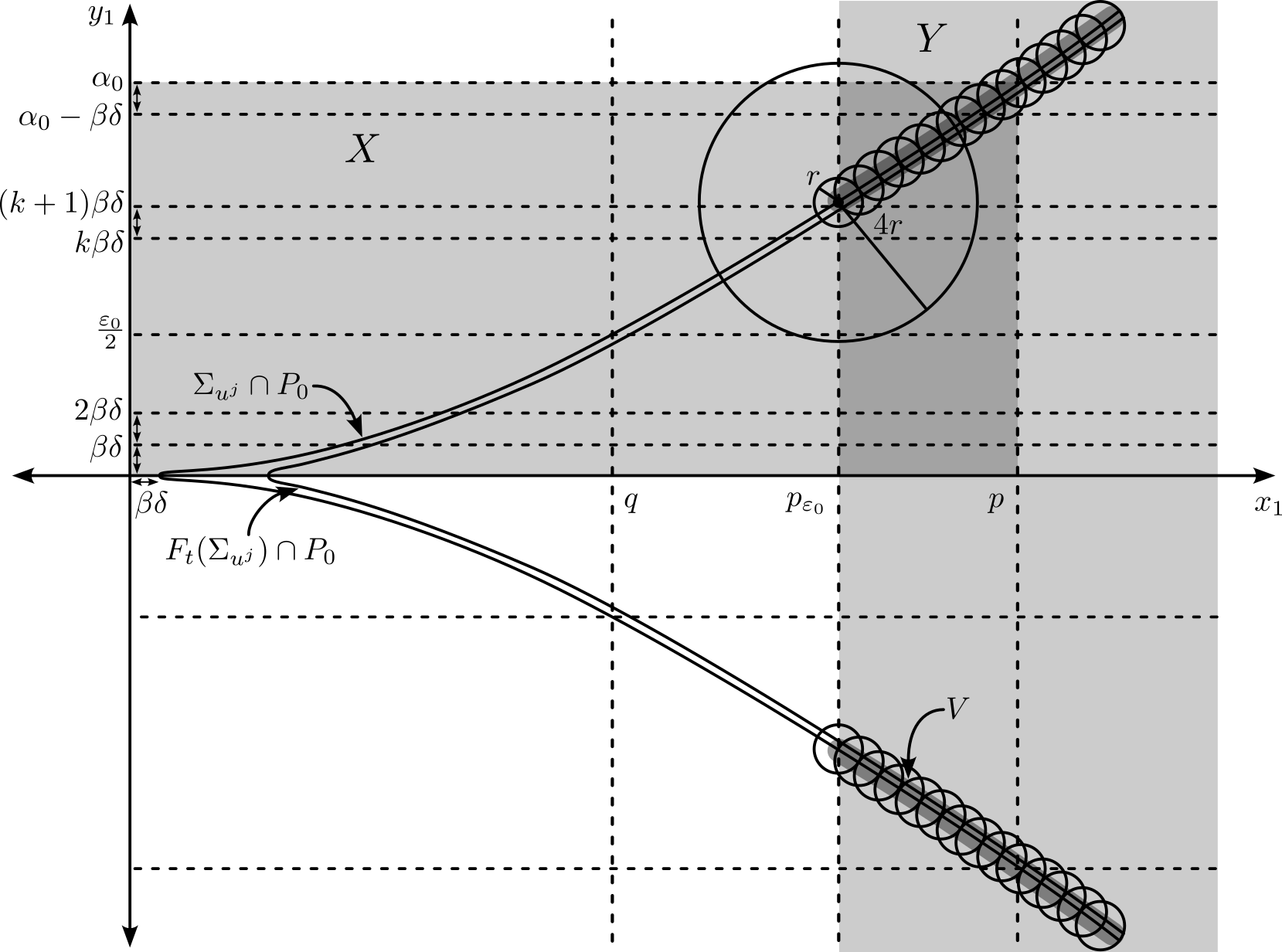}
        \caption{Regions of interest for Evans-Spruck and pseudolocality} 
        \label{fig:covering}
    \end{figure}

    By construction, each \(u^j(\cdot,t)\big|_{[0,p]}\) has a well defined inverse \(v^j:[0,\alpha_0)\to [0,p]\) for  \(t\in [\delta,T)\), where \(p\) is such that \(u^*(p) = \alpha_0\). Moreover, expression \eqref{eq:level_set_flow_containment} implies that \(v^j(y,t)\geq \beta t\) for \(y\in[0,\alpha_0)\) and \(t \in [\delta,T)\), which ensures that 
    \begin{align*}
        |x_2|<\beta\delta \leq \min \{v^j(y,t):y\in [0, \alpha_0), \ t\in [\delta, T), \ j\geq 1\}.
    \end{align*}
    Further restricting \(T\) if necessary so that \(\beta T< \alpha_0\), the above inequality tells us that a portion of \(F_t(\Sigma_{u^j})\) can be written as a graph \(\tilde{v}_0^j\) over the 3-ball \(B_{\beta\delta}(0)\cap \{x_1=0\}\) with \(M_0 \coloneqq \sup |\tilde{v}_0^j| \leq p\). Note in particular that the domain and \(M_0\) are independent of both \(t\) and \(j\). To cover all of \(F_t(\Sigma_{u^j})\cap X\), we can repeat this process to define graphs \(\{\tilde{v}_\ell^j\}_{\ell=1}^N\) on balls of radius \(\beta\delta\) centred at points \((0,0,\ell\beta\delta,0)\), and will also require one final ball centred at \((0,0,\alpha-\beta\delta,0)\) to cover any remaining gap. More formally, set \(N = \lfloor \frac{\alpha_0}{\beta\delta}\rfloor\) and for \(\ell\in \{0,\dots,N\}\) define
    \begin{align*}
        \tilde{v}^j_\ell:\ &B_{\beta\delta}(z_\ell)\cap \{x_1=0\}\times [\delta,T) \to [0,p]\\
        &(x_2,y_1,y_2,t)\mapsto\sqrt{\left(v^j\left(\sqrt{(y_1-\ell\beta\delta)^2 + y_2^2},\ t\right)\right)^2 - x_2^2}
    \end{align*}
    where \(z_\ell = (0,0,\ell\beta\delta,0)\). Then define
    \begin{align*}
        \tilde{v}^j_{N+1}:\ &B_{\beta\delta}(z_{N+1})\cap \{x_1=0\} \times [\delta, T) \to [0,p] \\
        &(x_2,y_1,y_2,t) \mapsto \sqrt{\left(v^j\left(\sqrt{(y_1-(\alpha_0-\beta\delta))^2 + y_2^2},\ t\right)\right)^2 - x_2^2}
    \end{align*}
    where \(z_{N+1} = (0,0,\alpha_0-\beta\delta,0)\). See Figure \ref{fig:covering} for an illustration.

    Since each \(\tilde{v}^j_\ell\) is smooth and solves the full graph equation \eqref{eq:full_graph_equation} in \(B_{\beta\delta}(z_k)\cap\{x_1=0\}\times[\delta,T)\), applying the Evans-Spruck estimate from Proposition \ref{prop:evans_spruck} with \(M_\ell \coloneqq \sup|\tilde{v}^j_k| \leq p\), then upgrading the regularity using Propostion \ref{prop:evans_spruck_schauder}, it follows that for any \(\varepsilon>0\), the first \(k\) derivatives of \(\tilde{v}^j_\ell\) are uniformly bounded on \(B_{\beta\delta-\varepsilon}(z_\ell)\cap\{x_1=0\}\times(\delta+\varepsilon^2,T)\), which implies uniform bounds on the first \(k-2\) derivatives of the second fundamental forms of these graphs, each of which cover a small portion of \(F_t(\Sigma_{u^j})\). Note that the balls over which these graphs are defined cover all of \(\{(\bar{x},\bar{y})\in P_0: x_1 = 0, \ \varepsilon - \beta \delta \leq y_1 \leq \alpha - \varepsilon\}\), and \(\tilde{v}^j_\ell(0,y,0,t) = v^j(y,t)\) for \(y\in B_{\beta\delta-\varepsilon}(z_\ell)\cap P_0\). In particular this implies that the first \(k-2\) derivatives of the second fundamental form of \(F_t(\Sigma_{u^j})\) are uniformly bounded at points in \(X\).

    \smallskip

    \noindent{\bf Case 2: Away from the origin.} We now prove uniform derivative bounds away from the origin by means of a pseudolocality argument. First, let \(q\) be such that \({u}^*(q)=\varepsilon_0/2\). By property \eqref{eq:approximation_property} of the approximation, we know that \(u^j(x) = {u}^*(x)\) for \(x\geq \delta_0 2^{-j}\), where \(\delta_0\) is a fixed constant depending only on \({u}^*\). Thus, up to discarding finitely many \(u^j\), it holds that \(u^j(x) = {u}^*(x)\) for all \(x\geq q\). Since \(\Sigma_{u^*}\) is smooth away from the origin, we can write \(\Sigma_{u^j} \cap \{(\bar{x},\bar{y})\in P_0: x_1\geq q\}\) as finitely many smooth graphs \(\tilde{u}_\ell\) over finitely many domains \(\Omega_\ell\) so that \(\|\tilde{u}_\ell\|_{C^4}(\Omega_\ell)\leq M\) for \(\ell \in \{1,\dots,L\}\). Crucially, these graphs are defined independently of \(j\).

    Now, given \(\beta_0\) as in Proposition \ref{prop:pseudolocality}, choose \(r\) small enough so that \(M\leq \beta_0r^{-3}\). Then it is possible to cover \(\Sigma_{u^j}\cap Y\) with 2-balls \(\{B_r(z_k)\cap P_0\}_{k=1}^N\) such that \(z_1 = (p_{\varepsilon_0},0,u^*(p_{\varepsilon_0}),0)\), \(z_N = (p_{\varepsilon_0},0,-u^*(p_{\varepsilon_0}),0)\) and \(P_0 \cap \bigcup_{k=1}^N B_r(z_k)\) contains a tubular neighbourhood of radius \(r/2\) of \(\Sigma_{u^j}\cap Y\), which we denote by \(V\). Possibly shrinking \(r\), we can ensure that the collection \(\{B_{4r}(z_i) \cap P_0\}_{k=1}^N\) consisting of 2-balls of radius \(4r\) still covers \(\Sigma_{u^j}\cap Y\) without touching two disconnected segments of the hypersurface (this is only a problem near the origin). See Figure \ref{fig:covering} for an illustration. Each \(\Sigma_{u^j}\cap B_{4r}(z_k)\) is then a restriction of the graph \(\tilde{u}_\ell\) to some \(\bar{\Omega}_{k,\ell} \subset \bar{\Omega}_\ell\), which satisfies \(\|\tilde{u}_\ell\big|_{\bar{\Omega}_{k,\ell}}\|_{C^4}\leq M\leq \beta_0/r^4\), for any \(k\in\{1,\dots,N\}\), up to slightly modifying the covering or the domains \(\bar{\Omega}_\ell\). 

    We can now apply Proposition \ref{prop:pseudolocality} to conclude that there exists \(C\) independent of \(j\) such that \(|A^j(w,t)|\leq Cr^{-1}\) for all \(t\leq \min\{r^2\beta_0,T\}\) and \(w\in F_t(\Sigma_{u^j})\cap B_r(z_k)\), for \(T>0\) determined in Case 1. To conclude a uniform second fundamental form bound in the cylindrical region \(Y\), we need to show that there exists a time \(t^*>0\) such that any point on \(F_t(\Sigma_{u^j})\cap Y\) lies in one of the 2-balls \(P_0\cap B_r(z_k)\) and hence in \(B_r(z_k)\). Symbolically, we require
    \begin{align*}
        Y \cap F_t(\Sigma_{u^j}) \subset P_0\cap {\textstyle\bigcup_{k=1}^N} B_r(z_k),
    \end{align*}
    for \(t<t^*\). To prove this, note that for any \(x\in P_0\) we have \(\dist(x,\Sigma_{u^j})=\dist(x,\Sigma_{u^j}\cap P_0)\), by the symmetry of \(\Sigma_{u^j}\), and by the avoidance principle, \(\dist(w,\Sigma_{u^j})\leq \sqrt{6t}\) for any \(w\in P_0\cap F_t(\Sigma_{u^j})\). If \(t^*\) is chosen so that \(\sqrt{6t^*}\leq r/2\), then \(\dist(w,\Sigma_{u^j}\cap P_0) \leq r/2\), so \(w\) belongs to the tubular neighbourhood \(V \subset \bigcup_{k=1}^N B_r(z_k)\) that we constructed earlier. Therefore the result holds for \(t\leq\min\{r^2\beta_0,\ r^2/24,\ T\}\). A uniform bound on higher derivatives of the second fundamental form can then be found using \cite[Proposition 3.22]{Ecker04}.
\end{proof}

We now complete the proof of instant smoothness.

\begin{prop}\label{prop:instant_smoothness}
    There is a \(T>0\) such that \(F_t(\Sigma_{u^*})\) is a smooth hypersurface for \(0<t<T\).
\end{prop}
\begin{proof}
    By Lemma \ref{lem:mean_convex_approximation}, \(\Sigma_{u^*}\) admits a mean convex approximation \(\Sigma_{u^j}\). It thus follows from Remark \ref{rem:mean_convex_limit} that \(F_t(\Sigma_{u^j}) \to F_t(\Sigma_{u^*})\) in the Hausdorff sense. On the other hand, Proposition \ref{prop:uniform_derivative_bounds} implies that that \(F_t(\Sigma_{u^j})\) converges uniformly to a smooth hypersurface for \(t\in(0,T)\), therefore by uniqueness of the level set flow the limits must agree and \(F_t(\Sigma_{u^*})\) is smooth. 
\end{proof}

This allows us to complete the proof of Theorem \ref{thm:main_theorem} in \(\R^4\).

\begin{proof}[Proof of \ref{thm:main_theorem}]
    We already established statement (1) in Section \ref{sec:topological_change_example}. Now let \(\Sigma_{u^*}\) be as before. By Proposition \ref{prop:instant_smoothness}, there exists \(T>0\) such that \(F_{t}(\Sigma_{u^*})\) is smooth for \(0<t<T\). Since \( F_{t}(\Sigma_{u^*})\) is smooth, closed and avoids \(P_2\) (by mean convexity of the flow), it is a smooth embedding of \(\Sb^1\times\Sb^2\). This completes the proof of statement (2).
\end{proof}

\begin{rem}[Extension to higher dimensions]\label{rem:higher_dimensions}
    Very few modifications are required to extend the proof of Theorem \ref{thm:main_theorem} from \(SO(2)\times SO(2)\) invariant flow in \(\R^4\) to \(SO(p)\times SO(q)\)-invariant flow in \(\R^{p+q}\). The most substantial modification is required in Section \ref{sec:singularity_formation}, where White's classification result allows for larger number of potential tangent flows that need to be excluded. However, the exact statement of Lemma \ref{lem:disc_containment} still holds in higher dimensions, with the only change being that the \(\ell\) becomes a \((p-1)\)-dimensional plane in \(P_1\), \(H_1^-\) and \(H_1^+\) both become \(p\)-dimensional half-planes in \(P_1\), and \(H_2^-\) becomes a \(q\)-dimensional plane in \(P_2\). This allows the additional tangent flows to be ruled out in bulk using an almost identical argument to the present one. Similar modifications will need to be made when considering the limit flow in Lemma \ref{lem:differentiability}.

    Aside from this, many of the standard results for mean curvature flow that are used in this paper involve constants that depend implicitly on the dimension, but not in a way that fundamentally changes any arguments. In particular, the rates of shrinking spheres and cylinders will need to be adjusted appropriately. The one notable adjustment on this front relates to the admissible range of \(\beta\) for the hyperboloid comparison principle (Lemma \ref{lem:hyperboloid_comparison}) which is not standard in the literature. In the general case, we require \(0\leq \beta \leq 1\) when \(p\leq q-1\), and \(2(p-q+1)/(2p-1)\leq \beta \leq 1\) when \(p>q-1\). 
\end{rem}